\documentclass[12pt,letterpaper,english,american,mathcal]{elsarticle}
\usepackage{bookman}
\usepackage{helvet}
\usepackage{courier}
\usepackage[T1]{fontenc}
\usepackage[latin9]{inputenc}
\usepackage{fancyhdr}
\usepackage{textcomp}
\usepackage{mathrsfs}
\usepackage{amsthm}
\usepackage{amsmath}
\usepackage{amssymb}
\usepackage[all]{xy}

\makeatletter

\pdfpageheight\paperheight
\pdfpagewidth\paperwidth

\numberwithin{equation}{section}
\numberwithin{figure}{section}
  \theoremstyle{plain}
  \newtheorem*{thm*}{\protect\theoremname}
  \theoremstyle{plain}
  \newtheorem*{cor*}{\protect\corollaryname}
\theoremstyle{plain}
\newtheorem{thm}{\protect\theoremname}[section]
  \theoremstyle{remark}
  \newtheorem{rem}[thm]{\protect\remarkname}
  \theoremstyle{definition}
  \newtheorem{defn}[thm]{\protect\definitionname}
  \theoremstyle{remark}
  \newtheorem*{rem*}{\protect\remarkname}
  \theoremstyle{definition}
  \newtheorem{example}[thm]{\protect\examplename}
  \theoremstyle{plain}
  \newtheorem{prop}[thm]{\protect\propositionname}
  \theoremstyle{plain}
  \newtheorem{lem}[thm]{\protect\lemmaname}
  \theoremstyle{plain}
  \newtheorem{cor}[thm]{\protect\corollaryname}

\usepackage{verbatim} 
\usepackage{eucal}
\usepackage{amsthm}
\usepackage[all]{xy}
\usepackage{rotating}
\usepackage{psfrag}
\usepackage{float}
\usepackage{xr}

\usepackage{eucal}
\usepackage{dsfont}
\SelectTips{cm}{}
\xyoption{line}

\makeatletter \newcommand{\xyR}[1]{%
\makeatletter \xydef@\xymatrixrowsep@{#1} \makeatother }
\makeatletter \newcommand{\xyC}[1]{%
\makeatletter \xydef@\xymatrixcolsep@{#1} \makeatother }

\SelectTips{cm}{}
\xyoption{line}
\mathcode`\:="603A 
\DeclareSymbolFont{rsfs}{U}{rsfs}{m}{n}
\DeclareSymbolFontAlphabet{\mathrf}{rsfs}

\makeatother

\usepackage{babel}
  \addto\captionsamerican{\renewcommand{\corollaryname}{Corollary}}
  \addto\captionsamerican{\renewcommand{\definitionname}{Definition}}
  \addto\captionsamerican{\renewcommand{\examplename}{Example}}
  \addto\captionsamerican{\renewcommand{\lemmaname}{Lemma}}
  \addto\captionsamerican{\renewcommand{\propositionname}{Proposition}}
  \addto\captionsamerican{\renewcommand{\remarkname}{Remark}}
  \addto\captionsamerican{\renewcommand{\theoremname}{Theorem}}
  \addto\captionsenglish{\renewcommand{\corollaryname}{Corollary}}
  \addto\captionsenglish{\renewcommand{\definitionname}{Definition}}
  \addto\captionsenglish{\renewcommand{\examplename}{Example}}
  \addto\captionsenglish{\renewcommand{\lemmaname}{Lemma}}
  \addto\captionsenglish{\renewcommand{\propositionname}{Proposition}}
  \addto\captionsenglish{\renewcommand{\remarkname}{Remark}}
  \addto\captionsenglish{\renewcommand{\theoremname}{Theorem}}
  \providecommand{\corollaryname}{Corollary}
  \providecommand{\definitionname}{Definition}
  \providecommand{\examplename}{Example}
  \providecommand{\lemmaname}{Lemma}
  \providecommand{\propositionname}{Proposition}
  \providecommand{\remarkname}{Remark}
  \providecommand{\theoremname}{Theorem}
\providecommand{\theoremname}{Theorem}

\begin{document}

\begin{frontmatter}{}

\title{Steenrod Coalgebras}

\author{Justin R. Smith}

\address{Department of Mathematics\\
Drexel University\\
Philadelphia, PA 19104}

\ead{jsmith@drexel.edu}

\ead[url]{http://vorpal.math.drexel.edu}

\global\long\def\ring{\mathscr{R}}

\global\long\def\rfrac{\mathbb{F}}

\global\long\def\field{\rfrac}

\global\long\def\integers{\mathbb{Z}}
\global\long\def\coend{\mathrm{CoEnd}}
 \global\long\def\coassoc{\mathrm{Coassoc}}
\global\long\def\homa{\mathrm{Hom}}
\global\long\def\zend{\mathrm{End}}
\global\long\def\rs#1{\mathrm{R}S_{#1 }}
\global\long\def\forgetful#1{\lceil#1\rceil}
\global\long\def\zs#1{\mathbb{Z}S_{#1 }}
\global\long\def\homzs#1{\mathrm{Hom}_{\mathbb{Z}S_{#1 }}}
\global\long\def\rings#1{\ring S_{#1}}
\global\long\def\zpi{\mathbb{Z}\pi}
\global\long\def\cf#1{C(#1 )}
\global\long\def\ddelta{\dot{\Delta}}
\global\long\def\dimlimiter{\triangleright}
\global\long\def\coalgcat{\mathrf S_{0}}
\global\long\def\ccoalgcat{\mathrf S_{\mathrm{cell}}}
\global\long\def\hcoalgcat{\mathrf{S}}
\global\long\def\ircoalgcat{\mathrf I_{0}}
\global\long\def\bircoalgcat{\mathrf{I}_{0}^{+}}
\global\long\def\hircoalgcat{\mathrf I}
\global\long\def\chaincat{\mathbf{Ch}}

\global\long\def\chaincatr{\chaincat_{0}}
\global\long\def\chaincatp{\chaincat_{+}}
\global\long\def\coalgcat{\mathrf S_{0}}
\global\long\def\bchaincat{\chaincatr}

\global\long\def\coll{\mathrm{Coll}}
\global\long\def\ilimit{\varprojlim\,}
\global\long\def\dlimit{\varinjlim\,}
\global\long\def\dilimit{{\varprojlim}^{1}\,}
\global\long\def\spaces{\mathbf{S}_{0}}

\global\long\def\coker{\mathrm{{coker}}}
\global\long\def\lcell{L_{\mathrm{cell}}}
\global\long\def\ccoalgcat{\mathrf S_{\mathrm{cell}}}
\global\long\def\coS{\mathbf{coS}}
\global\long\def\cocell{\mathbf{co}\ccoalgcat}
\global\long\def\realization{\mathcal{H}}
\global\long\def\totalspace#1{\mathrm{Tot}(#1)}
\global\long\def\totalcoalg#1{\mathrm{Tot}_{\ccoalgcat}(#1)}
\global\long\def\cofc#1{\mathbf{co}\fc#1}

\global\long\def\ccoalgcat{\mathrf S_{\mathrm{cell}}}
\global\long\def\hcoalgcat{\mathrf{S}}
\global\long\def\ircoalgcat{\mathrf I_{0}}
\global\long\def\bircoalgcat{\mathrf{I}_{0}^{+}}
\global\long\def\coforgetful#1{\mathbf{co}-\forgetful{#1}}
\global\long\def\cochaincat{\mathbf{co}\chaincat}
\global\long\def\zinfty#1{\integers_{\infty}#1}
\global\long\def\pgam{\tilde{\Gamma}}
\global\long\def\pz{\tilde{\ring}}

\global\long\def\sab{\mathbf{sAB}}

\global\long\def\sabz{\sab_{0}}

\global\long\def\homzp{\mathrm{Hom}_{\integers[C_{p}]}}
\global\long\def\finitecofree{\mathcal{F}_{\S}^{c}}
\global\long\def\rzp{\mathrm{R}(\integers/p\cdot\integers)}
\global\long\def\vp{\mathcal{V}_{p}}
\global\long\def\trivialcofree{\bar{\mathcal{F}_{\S}}}
\global\long\def\slength#1{|#1|}
\global\long\def\barcs{\bar{\mathcal{B}}}

\global\long\def\exp{\operatorname{exp}}

\global\long\def\ints{\mathbb{Z}}

\global\long\def\rats{\mathbb{Q}}

\global\long\def\qhat{\hat{Q}}

\global\long\def\moore#1{\{#1\}}

\global\long\def\s{\mathfrak{S}}

\global\long\def\tmap#1{\mathrm{T}_{#1}}

\global\long\def\Tmap#1{\mathfrak{T}_{#1}}

\global\long\def\glist#1#2#3{#1_{#2},\dots,#1_{#3}}

\global\long\def\blist#1#2{\glist{#1}1{#2}}

\global\long\def\enlist#1#2{\{\blist{#1}{#2}\}}

\global\long\def\tlist#1#2{\tmap{\blist{#1}{#2}}}

\global\long\def\Tlist#1#2{\Tmap{\blist{#1}{#2}}}

\global\long\def\fface{\mathrm{F}_{0}}

\global\long\def\lface{\tilde{\mathrm{F}}}

\global\long\def\nth#1{\mbox{#1}^{\mathrm{th}}}

\global\long\def\tunder#1#2{\tmap{\underbrace{{\scriptstyle #1}}_{#2}}}

\global\long\def\Tunder#1#2{\Tmap{\underbrace{{\scriptstyle #1}}_{#2}}}

\global\long\def\tunderi#1#2{\tunder{1,\dots,#1,\dots,1}{#2^{\mathrm{th}}\ \mathrm{position}}}

\global\long\def\Tunderi#1#2{\Tunder{1,\dots,#1,\dots,1}{#2^{\mathrm{th}}\,\mathrm{position}}}

\global\long\def\img{\operatorname{im}}

\global\long\def\ss{\mathbf{S}}

\global\long\def\ssz{\ss_{0}}

\global\long\def\ns#1{\mathcal{N}^{#1}}

\global\long\def\cfn#1{\mathcal{N}(#1)}

\global\long\def\m#1{\mathscr{M}_{#1}}

\global\long\def\freeop{\mathcal{F}}

\global\long\def\kerftos{\mathscr{K}}

\global\long\def\comm{\mathcal{C}\mathrm{omm}}

\global\long\def\steen{\mathscr{S}}

\global\long\def\arity{\operatorname{arity}\,}

\global\long\def\fc#1{\mathrm{hom}_{\steen}(\bigstar,#1)}

\global\long\def\dcat{\mathbf{D}}

\global\long\def\ords{\mathbf{\Delta}_{+}}

\global\long\def\sd{\mathfrak{f}}

\global\long\def\ds{\mathfrak{d}}

\global\long\def\homz{\mathrm{Hom}_{\ints}}

\begin{abstract}
This paper shows that a functorial version of the ``higher diagonal''
of a space used to compute \emph{Steenrod squares} actually contains
far more topological information --- including (in some cases) the
space's integral homotopy type.
\end{abstract}

\end{frontmatter}{}

\def\appendixname{}

\section{Introduction}

It is well-known that the Alexander-Whitney coproduct is functorial
with respect to simplicial maps. If $X$ is a simplicial set, $C(X)$
is the unnormalized chain-complex and $\rs 2$ is the \emph{bar-resolution}
of $\ints_{2}$ (see \cite{maclane:1975}), it is also well-known
that there is a unique homotopy class of $\ints_{2}$-equivariant
maps (where $\ints_{2}$ transposes the factors of the target) 
\[
\xi_{X}:\rs 2\otimes C(X)\to C(X)\otimes C(X)
\]
and that this extends the Alexander-Whitney diagonal. We will call
such structures, Steenrod coalgebras and the map $\xi_{X}$ the Steenrod
diagonal. In his construction of cup-$i$ products, Steenrod defined
a kind of dual of this map in \cite{steenrod-cup-i}.

With some care (see appendix~\ref{sec:Functorial-Steenrod-coalgebras}),
one can construct $\xi_{X}$ in a manner that makes it \emph{functorial}
with respect to simplicial maps although this is seldom done since
the \emph{homotopy class} of this map is what is generally studied.
Essentially, \cite{smith-cellular,smith-s-fundamental,smith:1994}
show that $C(X)$ possesses the structure of a functorial coalgebra
over an operad $\s$ (see example~\ref{exa:The-Barratt-Eccles-operad,})
and that the arity-2 portion of this operad-action is a functorial
version of $\xi_{X}$. Throughout this paper, we will assume this
\emph{functorial} version of $\xi_{X}$.

It is natural to ask whether $\xi_{X}$ encapsulates more information
about a topological space than its cup-product and Steenrod squares.
The present paper answers this question affirmatively for \emph{degeneracy-free}
simplicial sets. Roughly speaking, these are simplicial sets whose
degeneracies do not satisfy \emph{any relations} other than the minimal
set of identities all face- and degeneracy-operators must satisfy
--- see definition~\ref{def:degeneracy-free} and proposition~\ref{prop:intrinsic-degenracy-free}.
Every simplicial set is \emph{canonically} homotopy equivalent to
a degeneracy-free one (see proposition~\ref{prop:homotopy-equiv-sd-ds}).
The only place degeneracy-freeness is used in this paper is lemma~\ref{lem:degeneracy-free-inclusion}.
\begin{thm*}
\ref{cor:main-hurewicz} Let $X$ and $Y$ be pointed, reduced degeneracy-free
(see definition~\ref{def:degeneracy-free} in appendix~\ref{sec:Simplicial-sets-and})
simplicial sets with normalized chain-complexes $N(X)$ and $N(Y)$,
let $\ring=\ints_{p}$ for some prime $p$ or a subring of $\rats$,
and let 
\[
f:N(X)\otimes\ring\to N(Y)\otimes\ring
\]
 be a (purely algebraic) chain map that makes the diagram 
\begin{equation}
\xymatrix{{\rs 2\otimes N(X)\otimes\ring}\ar[r]^{1\otimes f}\ar[d]_{\xi_{X}\otimes1} & {\rs 2\otimes N(Y)\otimes\ring}\ar[d]^{\xi_{Y}\otimes1}\\
{N(X)\otimes\ring\otimes N(X)\otimes\ring}\ar[r]_{f\otimes f} & {N(Y)\otimes\ring\otimes N(Y)\otimes\ring}
}
\label{eq:coproduct-diagram}
\end{equation}
commute exactly (i.e., not merely up to a chain-homotopy). Then $f$
induces a simplicial map
\[
f_{\infty}:\ring_{\infty}X\to\ring_{\infty}Y
\]
where $\ring_{\infty}$ denotes the $\ring$-completion, (see \cite{bousfield-kan}
or \cite{goerss-jardine} for this concept) that makes the diagram\textup{\emph{
\[
\xymatrix{{X}\ar[d]_{\phi_{X}} & {Y}\ar[d]^{\phi_{Y}}\\
{\ring_{\infty}X}\ar[r]^{f_{\infty}}\ar[d]_{q_{X}} & {\ring{}_{\infty}Y}\ar[d]^{q_{Y}}\\
{\pz X}\ar[r]_{\Gamma f} & {\pz Y}
}
\]
commute.}} If $f$ is surjective, then $f_{\infty}$ is a fibration,
and if $f$ is also a homology equivalence, then $f_{\infty}$ is
a trivial fibration. \end{thm*}
\begin{cor*}
If $f$ is a surjective homology equivalence, $\ring=\ints$, and
$X$ and $Y$ are nilpotent then there exists a homotopy equivalence
\[
\bar{f}:|X|\to|Y|
\]
of topological realizations. Corollary~\ref{cor:final-result} states
that nilpotent, degeneracy-free spaces are homotopy equivalent if
and only if there exists a homology equivalence of their chain-complexes
that make diagram~\ref{eq:coproduct-diagram} commute.
\end{cor*}
Here, $\pz*$ is a pointed version of the $\ring$-free simplicial
abelian group functor --- see definitions~\ref{def:hurewicz-map}
and \ref{def:pz}.

Because of the canonical homotopy equivalence between all simplicial
sets and degeneracy-free ones, the result above implies:
\begin{cor*}
\ref{cor:statement-for-simplicial-sets} If $X$ and $Y$ are pointed
reduced simplicial sets and 
\[
f:C(X)\to C(Y)
\]
is a morphism of Steenrod coalgebras --- over unnormalized chain-complexes
--- then $f$ induces a commutative diagram
\[
\xymatrix{{X} & {Y}\\
\ds\circ\sd(X)\ar[d]_{\phi_{(\ds\circ\sd(X))}}\ar[u]^{g_{X}} & \ds\circ\sd(Y)\ar[d]^{\phi_{(\ds\circ\sd(Y))}}\ar[u]_{g_{Y}}\\
{\ring_{\infty}(\ds\circ\sd(X))}\ar[r]^{f_{\infty}}\ar[d]_{q_{(\ds\circ\sd(X))}} & {\ring_{\infty}(\ds\circ\sd(Y))}\ar[d]^{q_{(\ds\circ\sd(Y))}}\\
{\pz(\ds\circ\sd(X))}\ar[r]_{\pgam f} & {\pz(\ds\circ\sd(Y))}
}
\]
 where $g_{X}$ and $g_{Y}$ are homotopy equivalences if $X$ and
$Y$ are Kan complexes --- and homotopy equivalences of their topological
realizations otherwise. In particular, if $X$ and $Y$ are nilpotent,
$\ring=\ints$, and $f$ is an integral homology equivalence, then
the topological realizations $|X|$ and $|Y|$ are homotopy equivalent.
\end{cor*}
Here, $\sd$ and $\ds$ are functors defined in definition~\ref{def:sd-ds-functors}
in appendix~\ref{sec:Simplicial-sets-and}. Singular simplicial sets
are always Kan complexes.

The reader might wonder how the Steenrod diagonal can contain any
information beyond the structure of a space at the prime 2. The answer
is that it forms part of an operad structure that contains information
about all primes --- and the only part of this complex operad structure
needed to compute, for instance, Steenrod $p^{\text{th}}$ powers
is the Steenrod diagonal. 

For example, let $X$ be a simplicial set with functorial higher diagonal
\[
h:\rs 2\otimes C(X)\to C(X)\otimes C(X)
\]
let $\Delta=h([\,]\otimes*):C(X)\to C(X)\otimes C(X)$ --- the Alexander-Whitney
diagonal --- and let $\Delta_{2}=h([(1,2]\otimes*):C(X)\to C(X)\otimes C(X)$.
A straightforward calculation shows that 
\[
(1\otimes\Delta)\circ\Delta_{2}:C(X)\to C(X)^{\otimes3}
\]
has the property that 
\begin{eqnarray}
\partial\{(1\otimes\Delta)\circ\Delta_{2}\} & = & (1\otimes\Delta)\circ\partial\Delta_{2}\nonumber \\
 & = & (1\otimes\Delta)\circ\{(1,2)-1\}\Delta\nonumber \\
 & = & (1,2,3)(\Delta\otimes1)\circ\Delta-(1\otimes\Delta)\circ\Delta\nonumber \\
 & = & \{(1,2,3)-1\}(1\otimes\Delta)\circ\Delta\label{eq:prime3}
\end{eqnarray}
where $(1,2,3)$ is a cyclic permutation of the factors. It follows
that $\Delta$ and $\Delta_{2}$ incorporate information about $X$
at the prime 3. Although the argument in equation~\ref{eq:prime3}
is elementary, the author is unaware of any prior instance of it. 

This paper's general approach to homotopy theory is the end result
of a lengthy research program involving some of the 20th century's
leading mathematicians. In \cite{quillen:1969}, Daniel Quillen proved
that the category of simply-connected rational simplicial sets is
equivalent to that of commutative coalgebras over $\rats$. In \cite{sullivan.mm:1977},
Sullivan analyzed the algebraic and analytic properties of these coalgebras,
developing the concept of minimal models and relating them to de~Rham
cohomology. That work was dual to Quillen's and had the advantage
of being far more direct.

Since then, a major goal has been to develop a similar theory for
\emph{integral} homotopy types.

In \cite{smirnov:1985}, Smirnov asserted that the integral homotopy
type of a space is determined by a coalgebra-structure on its singular
chain-complex over an $E_{\infty}$-operad. Smirnov's proof was somewhat
opaque and several people known to the author even doubted the result's
validity. In any case, the $E_{\infty}$-operad involved was complex,
being uncountably generated in all dimensions.

In \cite{smith:1994}, the author showed that the chain-complex of
a space was naturally a coalgebra over an $E_{\infty}$-operad $\s$
and that this could be used to iterate the cobar construction. The
paper \cite{smith:1998-2} applied those results to show that this
$\s$-coalgebra determined the integral homotopy type of a simply-connected
space.

In \cite{mandell-einfty}\footnote{Based on Mandell's 1997 thesis.},
Mandell showed that the mod-$p$ cochain complex of a $p$-nilpotent
space had a algebra structure over an operad that determined the space's
$p$-type. In \cite{mandell-finite}, Mandell showed that the cochains
of a nilpotent space whose homotopy groups are all \emph{finite} have
an algebra structure over an operad that determined its integral homotopy
type.

The paper \cite{smith-cellular} showed that the $\s$-coalgebra structure
of a chain-complex had a ``transcendental'' structure that determines
a nilpotent space's homotopy type (without the finiteness conditions
of \cite{mandell-finite}). It essentially reprised the main result
of \cite{smith:1998-2}, using a very different proof-method. The
present paper shows that this transcendental structure even manifests
in the sub-operad of $\s$ generated by its arity-2 component, $\rs 2$. 

I am indebted to Dennis Sullivan for several interesting discussions.

\section{Definitions\label{sec:Definitions}}

Given a simplicial set, $X$, $C(X)$ will always denote its \emph{unnormalized}
chain-complex and $N(X)$ its \emph{normalized} one (with degeneracies
divided out).
\begin{rem}
\label{assu:rdef}Throughout this paper $\ring$ denotes a fixed ring
satisfying
\[
\ring=\begin{cases}
\ints_{p} & \text{for some prime }p\text{ or}\\
\ring\subset\rats
\end{cases}
\]
\end{rem}
\begin{defn}
\label{def:chaincat} We will denote the category of $\ring$-free
chain chain-complexes by $\chaincat$ and ones that are\emph{ bounded
from below} in dimension $0$ by $\bchaincat$.
\end{defn}
We make extensive use of the Koszul Convention (see~\cite{gugenheim:1960})
regarding signs in homological calculations:
\begin{defn}
\label{def:koszul} If $f:C_{1}\to D_{1}$, $g:C_{2}\to D_{2}$ are
maps, and $a\otimes b\in C_{1}\otimes C_{2}$ (where $a$ is a homogeneous
element), then $(f\otimes g)(a\otimes b)$ is defined to be $(-1)^{\deg(g)\cdot\deg(a)}f(a)\otimes g(b)$. \end{defn}
\begin{rem}
If $f_{i}$, $g_{i}$ are maps, it isn't hard to verify that the Koszul
convention implies that $(f_{1}\otimes g_{1})\circ(f_{2}\otimes g_{2})=(-1)^{\deg(f_{2})\cdot\deg(g_{1})}(f_{1}\circ f_{2}\otimes g_{1}\circ g_{2})$.
\end{rem}
The set of morphisms of chain-complexes is itself a chain complex:
\begin{defn}
\label{def:homcomplex}Given chain-complexes $A,B\in\chaincat$ define
\[
\homz(A,B)
\]
to be the chain-complex of graded $\ring$-morphisms where the degree
of an element $x\in\homz(A,B)$ is its degree as a map and with differential
\[
\partial f=f\circ\partial_{A}-(-1)^{\deg f}\partial_{B}\circ f
\]
As a $\ring$-module $\homz(A,B)_{k}=\prod_{j}\homz(A_{j},B_{j+k})$.\end{defn}
\begin{rem*}
Given $A,B\in\chaincat^{S_{n}}$, we can define $\homzs n(A,B)$ in
a corresponding way.
\end{rem*}
Recall the concept of \emph{algebraic operad} in \cite{operad-loday}
or \cite{kriz-may}: a sequence of $\zs n$ chain-complexes $\{\mathcal{V}(n)\}$
for $n\ge0$ with structure maps
\[
\gamma_{i_{1},\dots,i_{n}}:\mathcal{V}(n)\otimes\mathcal{V}(i_{1})\otimes\cdots\otimes\mathcal{V}(i_{n})\to\mathcal{V}(i_{1}+\cdots i_{n})
\]
for $n,i_{1},\dots,i_{n}\ge0$.
\begin{defn}
We will call the operad $\mathcal{V}=\{\mathcal{V}(n)\}$ $\Sigma$-\emph{cofibrant}
if $\mathcal{V}(n)$ is $\zs n$-projective for all $n\ge0$.\end{defn}
\begin{rem*}
The operads we consider here correspond to \emph{symmetric} operads
in \cite{smith:cofree}.

The term ``unital operad'' is used in different ways by different
authors. We use it in the sense of Kriz and May in \cite{kriz-may},
meaning the operad has a $0$-component that acts like an arity-lowering
augmentation under compositions. Here $\mathcal{V}(0)=\ring$.

The term $\Sigma$-\emph{cofibrant} first appeared in \cite{berger-moerdijk-axiom-operad}.
\end{rem*}
We also need to recall compositions in operads:
\begin{defn}
\label{def:operad-comps}If $\mathcal{V}$ is an operad with components
$\mathcal{V}(n)$ and $\mathcal{V}(m)$, define the $i^{\text{th}}$
\emph{composition,} with $1\le i\le n$
\[
\circ_{i}:\mathcal{V}(n)\otimes\mathcal{V}(m)\to\mathcal{V}(n+m-1)
\]
by
\[
\xymatrix{{\mathcal{V}(n)\otimes\mathcal{V}(m)}\ar@{=}[d]\\
{\mathcal{V}(n)\otimes\ints^{i-1}\otimes\mathcal{V}(m)\otimes\ints^{n-i}}\ar[d]^{1\otimes\eta^{i-1}\otimes1\otimes\eta^{n-i}}\\
{\mathcal{V}(n)\otimes{\mathcal{V}(1)}^{i-1}\otimes\mathcal{V}(m)\otimes{\mathcal{V}(1)}^{n-i}}\ar[d]^{\gamma}\\
{\mathcal{V}(n+m-1)}
}
\]
Here $\eta:\ints\to\mathcal{V}(1)$ is the unit.\end{defn}
\begin{rem*}
Operads were originally called \emph{composition algebras} and defined
in terms of these operations --- see \cite{gerstenhaber:1962}. 
\end{rem*}
It is well-known that the compositions and the operad structure-maps
determine each other --- see definition~2.12 and proposition~2.13
of \cite{smith:cofree}.

A simple example of an operad is:
\begin{example}
\label{example:frakS0}For each $n\ge0$, $\s_{0}(n)=\zs n$, with
structure-map a $\ints$-linear extension of
\[
\gamma_{\alpha_{1},\dots,\alpha_{n}}:S_{n}\times S_{\alpha_{1}}\times\cdots\times S_{\alpha_{n}}\to S_{\alpha_{1}+\cdots+\alpha_{n}}
\]
defined by
\[
\gamma_{\alpha_{1},\dots,\alpha_{n}}(\sigma\times\theta_{1}\times\cdots\times\theta_{n})=\tlist{\alpha}n(\sigma)\circ\left(\theta_{1}\oplus\cdots\oplus\theta_{n}\right)
\]
with $\sigma\in S_{n}$ and $\theta_{i}\in S_{\alpha_{i}}$ where
$\tlist{\alpha}n(\sigma)\in S_{\sum\alpha_{i}}$ is a permutation
that permutes the $n$ blocks
\begin{multline*}
\{1,\dots,\alpha_{1}\},\{\alpha_{1}+1,\alpha_{1}+\alpha_{2}\},\dots,\\
\{\alpha_{1}+\cdots+\alpha_{n-1}+1,\alpha_{1}+\cdots+\alpha_{n}\}
\end{multline*}
 via $\sigma$. See \cite{smith:1994} for explicit formulas and computations.
\end{example}
Another important operad is:
\begin{example}
\label{exa:The-Barratt-Eccles-operad,}The operad, $\mathfrak{S}$,
defined in \cite{smith:1994} is given by $\mathfrak{S}(n)=\rs n$
--- the normalized\emph{ bar-resolution} of $\ints$ over $\zs n$.
This is well-known (like the closely-related Barrett-Eccles operad
defined in \cite{barratt-eccles-operad}) to be a Hopf-operad, i.e.
equipped with an operad morphism
\[
\delta:\mathfrak{S}\to\mathfrak{S}\otimes\mathfrak{S}
\]
and is important in topological applications. See \cite{smith:1994}
for formulas for the structure maps.
\end{example}
For the purposes of this paper, the main example of an operad is
\begin{defn}
\label{def:coend}Given any $C\in\chaincat$, the associated \emph{coendomorphism
operad}, $\coend(C)$ is defined by
\[
\coend(C)(n)=\homz(C,C^{\otimes n})
\]
 Its structure map
\begin{multline*}
\gamma_{\alpha_{1},\dots,\alpha_{n}}:\homz(C,C^{\otimes n})\otimes\homz(C,C^{\otimes\alpha_{1}})\otimes\cdots\otimes\homz(C,C^{\otimes\alpha_{n}})\to\\
\homz(C,C^{\otimes\alpha_{1}+\cdots+\alpha_{n}})
\end{multline*}
simply composes a map in $\homz(C,C^{\otimes n})$ with maps of each
of the $n$ factors of $C$. 

This is a non-unital operad, but if $C\in\chaincat$ has an augmentation
map $\varepsilon:C\to\ring$ then we can  regard $\varepsilon$ as
the generator of $\coend(C)(0)=\ring\cdot\varepsilon\subset\homz(C,C^{\otimes0})=\homz(C,\ring)$.
\end{defn}
We use the coendomorphism operad to define the main object of this
paper:
\begin{defn}
\label{def:coalg}A \emph{coalgebra over an operad} $\mathcal{V}$
is a chain-complex $C\in\chaincat$ with an operad morphism $\alpha:\mathcal{V}\to\coend(C)$,
called its \emph{structure map.} We will sometimes want to define
coalgebras using the \emph{adjoint structure map, }
\begin{equation}
\alpha:C\to\prod_{n\ge0}\homzs n(\mathcal{V}(n),C^{\otimes n})\label{eq:structure-map}
\end{equation}
where $S_{n}$ acts on $C^{\otimes n}$ by permuting factors or the
set of chain-maps
\[
\alpha_{n}:C\to\homzs n(\mathcal{V}(n),C^{\otimes n})
\]
for all $n\ge0$ or even 
\[
\beta_{n}:\mathcal{V}(n)\otimes C\to C^{\otimes n}
\]

\end{defn}
It is not hard to see how \emph{compositions} (in definition~\ref{def:operad-comps})
relate to coalgebras
\begin{prop}
\label{prop:composition-coalgebra}If the maps $\beta_{n}:\mathcal{V}(n)\otimes C\to C^{\otimes n}$
for all $n\ge0$ define a coalgebra over an operad $\mathcal{V}$,
for any $x\in\mathcal{V}(n)$ and any $n\ge0$ define 
\[
\Delta_{x}=\beta_{n}(x\otimes*):C\to C^{\otimes n}
\]
If $x\in\mathcal{V}(n)$ and $y\in\mathcal{V}(m)$, then
\[
\Delta_{y\circ_{i}x}=\underbrace{1\otimes\cdots\otimes1\otimes\Delta_{y}\otimes1\otimes\cdots\otimes}_{i^{\text{th}}\text{ position}}\circ\Delta_{x}
\]
\end{prop}
\begin{proof}
Immediate, from definitions~\ref{def:operad-comps} and \ref{def:coend}.
\end{proof}

\subsection{Coalgebras over operads}
\begin{example}
Coassociative coalgebras are precisely the coalgebras over $\mathfrak{S}_{0}$
(see \ref{example:frakS0}). \end{example}
\begin{defn}
\label{def:coassoc}$\comm$ is an operad defined to have one basis
element $\{b_{i}\}$ for each integer $i\ge0$. Here the arity of
$b_{i}$ is $i$ and the degree is 0 and the these elements satisfy
the composition-law: $\gamma(b_{n}\otimes b_{k_{1}}\otimes\cdots\otimes b_{k_{n}})=b_{K}$,
where $K=\sum_{i=1}^{n}k_{i}$. The differential of this operad is
identically zero. The symmetric-group actions are trivial. \end{defn}
\begin{example}
Coassociative, commutative coalgebras are the coalgebras over $\comm$.
\end{example}
We can define a concept dual to that of a free algebra generated by
a set: 
\begin{defn}
\label{def:cofreecoalgebra}Let $D$ be a coalgebra over an operad
$\mathcal{V}$, equipped with a $\chaincat$-morphism $\varepsilon:\forgetful D\to E$,
where $E\in\chaincat$. Then $D$ is called \emph{the cofree coalgebra
over} $\mathcal{V}$ \emph{cogenerated} \emph{by} $\varepsilon$ if
any morphism in $\chaincat$
\[
f:\forgetful C\to E
\]
where $C\in\coalgcat$, induces a \emph{unique} morphism in $\coalgcat$
\[
\alpha_{f}:C\to D
\]
that makes the diagram 
\[
\xyC{40pt}\xymatrix{{\forgetful C}\ar[r]^{\forgetful{\alpha_{f}}}\ar[rd]_{f} & {\forgetful D}\ar[d]^{\varepsilon}\\
{} & {E}
}
\]
\emph{ }

Here $\alpha_{f}$ is called the \emph{classifying map} of $f$. If
$C\in\coalgcat$ then 
\[
\alpha_{f}:C\to L_{\mathfrak{\mathcal{V}}}\forgetful C
\]
 will be called the \emph{classifying map of} $C$.\end{defn}
\begin{rem}
\label{rem:cofree-property}This universal property of cofree coalgebras
implies that they are \emph{unique} up to isomorphism if they exist. 

The paper \cite{smith:cofree} explicitly constructs cofree coalgebras
for many operads:
\begin{itemize}
\item $L_{\mathcal{V}}C$ is the \emph{general} cofree coalgebra over the
operad $\mathcal{V}$ --- here, $C$, is a chain-complex that is not
necessarily concentrated in nonnegative dimensions. Then \cite{smith:cofree}
constructs $D=L_{\mathcal{V}}E$ as the maximal submodule of 
\[
\prod_{n=1}^{\infty}\homzs n(\mathcal{V}(n),E^{\otimes n})
\]
on which the dual of the structure-maps of $\mathcal{V}$ define a
coalgebra-structure: let $\iota:D\to\prod_{n=1}^{\infty}\homzs n(\mathcal{V}(n),E^{\otimes n})$
be the inclusion of chain-complexes. In the notation of definition~\ref{def:cofreecoalgebra},
an $\mathcal{V}$-coalgebra, $C$, is defined by its \emph{structure
map }(see equation~\ref{eq:structure-map})
\[
s:C\to\prod_{n=1}^{\infty}\homzs n(\mathcal{V}(n),C^{\otimes n})
\]
and its \emph{classifying map} $\alpha_{f}:D\to L_{\mathcal{V}}C$
is the coalgebra morphism defined by the diagram 
\begin{equation}
\xymatrix{{C}\ar[r]^{s\qquad\qquad\quad}\ar[d]_{\alpha_{f}} & {\prod_{n=1}^{\infty}\homzs n(\mathcal{V}(n),C^{\otimes n})}\ar[d]^{\prod_{n=1}^{\infty}\homzs n(1,f^{\otimes n})}\\
{D}\ar[r]_{\iota\qquad\qquad\quad} & {\prod_{n=1}^{\infty}\homzs n(\mathcal{V}(n),E^{\otimes n})}
}
\label{eq:classifying-map}
\end{equation}
An inductive argument shows that this is the \emph{unique} coalgebra
morphism compatible with $f$.
\end{itemize}
\end{rem}
In all cases, definition~\ref{def:cofreecoalgebra} implies the existence
of an adjunction
\begin{equation}
\forgetful *:\bchaincat\leftrightarrows:L_{\mathcal{V}}*\label{eq:cofree-adjunction}
\end{equation}
 where $\forgetful *:\steen\to\bchaincat$ is the \emph{forgetful
functor} from coalgebras to chain-complexes.

\section{Steenrod coalgebras\label{sec:Steenrod-coalgebras}}

We begin with:
\begin{defn}
\label{def:Steenrod-coalgebra}A \emph{Steenrod coalgebra,} $(C,\delta)$
is a chain-complex $C\in\chaincat$ equipped with a $\ints_{2}$-equivariant
chain-map
\[
\delta:\rs 2\otimes C\to C\otimes C
\]
 where $\ints_{2}$ acts on $C\otimes C$ by swapping factors and
$\rs 2$ is the bar-resolution of $\ints$ over $\zs 2$. A morphism
$f:(C,\delta_{C})\to(D,\delta_{D})$ is a chain-map $f:C\to D$ that
makes the diagram
\[
\xyR{30pt}\xymatrix{{C}\ar[r]^{f}\ar[d]_{\delta_{C}} & {D}\ar[d]^{\delta_{D}}\\
{C\otimes C}\ar[r]_{f\otimes f} & {D\otimes D}
}
\]
commute.
\end{defn}
Steenrod coalgebras are very general --- the underlying coalgebra
need not even be coassociative. The category of Steenrod coalgebras
is denoted $\steen$.
\begin{defn}
\label{def:f-operad}Let,\emph{ $\freeop$, }denote the \emph{free
operad} generated by\emph{ $\rs 2$.}\end{defn}
\begin{rem*}
See sections~5.2 and 5.5 of \cite{operad-loday} or section~5.8
of \cite{berger-moerdijk-axiom-operad} for an explicit construction
of $\freeop$. For instance
\[
\freeop(3)=\rs 2\otimes_{\zs 2}\left(\underbrace{\zs 3\otimes_{\zs 2}\rs 2}_{S_{2}\text{ generated by }(1,2)}\oplus\underbrace{\zs 3\otimes_{\zs 2}\rs 2}_{S_{2}\text{ generated by }(2,3)}\right)
\]
where $S_{2}=\ints_{2}$ swaps the summands and $\zs 3$ acts on $\freeop(3)$
by acting on the factors $\zs 3$ inside the parentheses.\end{rem*}
\begin{prop}
\label{pro:morphism-f-to-s}The identity map of $\rs 2$ uniquely
extends to an operad-morphism
\[
\xi:\freeop\to\s
\]
and the kernel is an operadic ideal (see section~5.2.16 of \cite{operad-loday})
denoted $\kerftos$. \end{prop}
\begin{rem*}
The image, $\xi(\freeop)\subset\s$, is the suboperad generated by
$\s(2)=\rs 2$. \end{rem*}
\begin{proof}
All statements follow immediately from the defining property of free
operads.
\end{proof}
Although the construction of $\freeop$ is fairly complex, it is easy
to describe \emph{coalgebras} over $\freeop$:
\begin{prop}
\label{prop:steenrod-iso-free}The category of coalgebras over $\freeop$
is identical to that of Steenrod coalgebras.\end{prop}
\begin{proof}
If $C$ is an $\freeop$-coalgebra then there exists a $\zs 2$-morphism
\[
\freeop(2)\otimes C=\rs 2\otimes C\to C\otimes C
\]
so $C$ is a Steenrod coalgebra. If $C$ is a Steenrod coalgebra,
it has an adjoint structure map
\[
\rs 2\to\homz(C,C\otimes C)=\coend(C)(2)
\]
 that \emph{uniquely} extends to an operad-morphism
\[
\freeop\to\coend(C)
\]
It is also clear that this correspondence respects morphisms.
\end{proof}
This has a number of interesting consequences:
\begin{thm}
\label{thm:universal-steenrod}If $C$ is a chain-complex, there exists
a universal Steenrod coalgebra $L_{\freeop}C$ --- the cofree coalgebra
over $\freeop$ cogenerated by $C$ --- equipped with a chain-map
\[
\varepsilon:L_{\freeop}C\to C
\]
with the property that, given any Steenrod coalgebra $D$ and any
chain-map $f:D\to C$, there exists a unique morphism of Steenrod
coalgebras
\[
\bar{f}:D\to L_{\freeop}C
\]
 that makes the diagram 
\[
\xymatrix{{D}\ar[r]^{\bar{f}}\ar[rd]_{f} & {L_{\freeop}C}\ar[d]^{\varepsilon}\\
{} & {C}
}
\]
 commute.\end{thm}
\begin{proof}
The conclusions are nothing but the defining properties of a cofree
coalgebra over $\freeop$. So the result follows immediately from
proposition~\ref{prop:steenrod-iso-free}.
\end{proof}

\section{The Dold-Kan functor and variants\label{sec:The-Dold-Kan-functor}}

Recall remark~\ref{assu:rdef} regarding the ring $\ring$. We recount
classic results regarding simplicial abelian groups:
\begin{defn}
\label{def:simplicial-abelian-groups}Let $\sab$ denote the category
of \emph{simplicial abelian groups} and $\sabz\subset\sab$ the full
subcategory of $\ints$-free pointed, reduced simplicial abelian groups.
If $A\in\sab$, let
\begin{enumerate}
\item $\moore A$ denote the \emph{Moore complex} of $A$ --- a (not necessarily
$\ints$-free chain complex made up of the simplices of $A$). Since
$A$ is a simplicial abelian group, every linear combination of simplices
of $A$ is \emph{also} a simplex of $A$ and the \emph{elements} of
$\moore A$ are the \emph{simplices} of $A$.
\item $NA\subset\moore A$ denote the \emph{normalized} chain-complex of
$A$ defined by
\[
NA_{n}=\bigcap_{i=0}^{n-1}\ker d_{i}\subset A_{n}
\]
where $d_{i}:A_{n}\to A_{n-1}$ are the face-operators. The boundary
is defined by $\partial_{n}=(-1)^{n}d_{n}:NA_{n}\to NA_{n-1}$. 
\end{enumerate}

The \emph{Dold-Kan functor} from the category of arbitrary chain complexes
(not necessarily $\ints$-free) to $\sab$ is denoted $\Gamma$.

\end{defn}
\begin{rem*}
It is well-known (see \cite[chapter III]{goerss-jardine}) that
\[
\pi_{i}(A)=H_{i}(\moore A)
\]
for $i\ge0$.
\end{rem*}
Given a simplicial set, we can construct a simplicial abelian group
from it:
\begin{defn}
\label{def:hurewicz-map}If $X$ is a simplicial set and $\ring$
is a ring following the conditions in remark~\ref{assu:rdef}, $\ring X$
denotes the $\ring$-free simplicial abelian group generated by $X$.
The \emph{Hurewicz map}
\[
h_{X}:X\to\ring X
\]
 sends a simplex $x\in X$ to $1\cdot x\in\ring X$.\end{defn}
\begin{rem*}
It is well-known (see \cite{bousfield-kan}, chapter~I, §~2), that
\[
\pi_{i}(\ring X)\cong H_{i}(X,\ring)
\]
 and that the Hurewicz map induces
\[
\pi_{i}(h_{X}):\pi_{i}(X)\to\pi_{i}(\ring X)=H_{i}(X,\ring)
\]
--- the Hurewicz homomorphism from homotopy groups to homology groups
(hence the name ``Hurewicz map'').
\end{rem*}
We have the classic Dold-Kan results (see \cite{kan-dold-kan} and
\cite[chapter III]{goerss-jardine}, corollary~2.3, theorem~2.5,
and corollary~2.12.):
\begin{prop}
\label{prop:dold-kan}For any chain-complex, $C$, concentrated in
nonnegative dimensions 
\[
N\Gamma C\cong C
\]
and for any simplicial abelian group $A$
\[
\Gamma NA\cong A
\]
These correspondences define an equivalence of categories between
the category of (not necessarily $\ints$-free) chain-complexes concentrated
in nonnegative dimensions and simplicial abelian groups.

In addition, there is an adjunction
\begin{equation}
N(*):\chaincatr\leftrightarrows\ss:\Gamma*\label{eq:chain-gamma-adjunction}
\end{equation}
where $N(*)$ is the (normalized) integral chain-complex functor (see
\cite[chapter~III]{goerss-jardine}). The functor $\Gamma*$ is defined
by
\begin{equation}
\Gamma(C)_{m}=\bigoplus_{\mathbf{m}\twoheadrightarrow\mathbf{n}}C_{n}\label{eq:dold-kan-description}
\end{equation}
for all $m>n\ge0$, where $\mathbf{m}$ and $\mathbf{n}$ are objects
of the ordinal number category $\mathbf{\Delta}$.\end{prop}
\begin{rem*}
Equation~\ref{eq:dold-kan-description} simply says that one forms
all possible ``formal'' degeneracies of $C$ and defines face and
degeneracy operators to via the defining identities that they satisfy.
\end{rem*}
We define a \emph{pointed} variant:
\begin{defn}
\label{def:pgamma}If $C\in\chaincatr$, then $C=C^{+}\oplus\ints_{0}$,
where $\ints_{0}$ is concentrated in dimension $0$ and equal to
$\ints$ there. We define
\[
\pgam C=\Gamma C/\Gamma\ints_{0}\cong\Gamma C^{+}
\]
--- where $/$ denotes a quotient-group.\end{defn}
\begin{rem*}
\label{rem:pgamequalsgam}Since $C=C^{+}\oplus\ints_{0}$, the equivalence
of categories in proposition\ref{prop:dold-kan} implies that 
\[
\Gamma C=\pgam C\times\Gamma\ints_{0}
\]
If $C\in\chaincatp$, note that $\pgam C=\Gamma C$, since $\moore{\Gamma C}_{0}=0$
--- the \emph{trivial} abelian group.

The reason we need a pointed variant is that the Bousfield-Kan cosimplicial
resolution (see \cite{bousfield-kan}) of a space requires it to be
pointed.
\end{rem*}
We also have the free abelian group functor (denoted $FA*$ in \cite{kan-dold-kan}):
\begin{defn}
\label{def:pz}If $X\in\ss$ and $\ring$ is a ring satisfying remark~\ref{assu:rdef},
$\ring X$ is the \emph{$\ring$-free simplicial abelian group} generated
by the simplices of $X$. If $X\in\ssz$ --- i.e., if $X$ is pointed
and reduced --- then we have the Bousefield-Kan \emph{pointed} version
of the free abelian group functor (see \cite{bousfield-kan}), the
quotient 
\[
\pz X=\ring X/\ring*
\]
 where $*$ is the sub-simplicial set generated by the basepoint of
$X$.\end{defn}
\begin{prop}
\label{prop:chain-maps-induce-zs}If $X$ is a simplicial set and
$N(X)$ is the normalized integral chain-complex of $X$ then
\[
\ring X=\Gamma(N(X)\otimes\ring)
\]
 If $X$ is pointed and reduced then
\[
\ring X=\pz X\times\ring*=\pgam\left(N(X)\otimes\ring\right)\times\Gamma\ring_{0}
\]
If $X$ and $Y$ are pointed reduced simplicial sets and 
\[
f:N(X)\otimes\ring\to N(Y)\otimes\ring
\]
 is a chain-map of normalized chain-complexes, then $f$ induces
\[
\pz^{n-1}\pgam f:\pz^{n-1}\pgam\left(N(X)\otimes\ring\right)=\pz^{n}X\to\pz^{n-1}\pgam\left(N(Y)\otimes\ring\right)=\pz^{n}Y
\]
for all $n>0$.\end{prop}
\begin{proof}
This follows immediately from the Dold-Kan results in proposition~\ref{prop:dold-kan}.
\end{proof}

\section{\label{sec:main-result}The Main result}

We begin with a definition:
\begin{defn}
\label{def:simplicial-sets}Let $\ss$ denote the category of \emph{simplicial
sets} and $\ssz$ that of \emph{pointed, reduced} simplicial sets.
\end{defn}
By following the procedure in appendix~\ref{sec:Functorial-Steenrod-coalgebras},
we get:
\begin{prop}
\label{prop:cf-intro}If $X\in\ss$ is a simplicial set, then the
unnormalized chain complex of $X$, $C(X)$ has a natural Steenrod
coalgebra structure and there exists a functor
\[
\cf *:\ss\to\steen
\]
from the category of simplicial sets to that of Steenrod coalgebras
concentrated in nonnegative dimensions. This projects to a Steenrod
coalgebra structure on the normalized chain-complex, $N(X)$.\end{prop}
\begin{proof}
See appendix~\ref{sec:Functorial-Steenrod-coalgebras} and proposition~\ref{prop:cf-functor}
for the details.
\end{proof}
Recall the concept of \emph{degeneracy-free} simplicial sets in definition~\ref{def:degeneracy-free}
in appendix~\ref{sec:Simplicial-sets-and}. The main (only?) reason
we are interested in them is:
\begin{lem}
\label{lem:degeneracy-free-inclusion}If $X$ is a degeneracy-free
simplicial set, its nondegenerate simplices form a delta-complex,
$\bar{X}$, and there is a natural inclusion
\[
\bar{X}\to\sd(X)
\]
 inducing an inclusion of Steenrod coalgebras
\[
\iota:N(X)=N(\bar{X})\to N(\sd(X))=C(X)
\]
\end{lem}
\begin{rem*}
See definition~\ref{def:sd-ds-functors} in appendix~\ref{sec:Simplicial-sets-and}
for the functor $\sd$. 

Although \emph{all} simplicial sets have an inclusion of \emph{chain-complexes}
\[
N(X)\to C(X)
\]
 the Steenrod coalgebra structure of $N(X)$ is defined as a \emph{quotient
}of that of $C(X)$ by the degenerate simplices. It follows that this
inclusion of chain-complexes does not necessarily imply one of Steenrod
coalgebras.\end{rem*}
\begin{proof}
This follows from proposition~\ref{prop:cx-is-nfx}.
\end{proof}
We also define
\begin{defn}
\label{def:pcmap}If $X$ is a pointed, reduced simplicial set with
unnormalized chain-complex $C(X)$ and $\ring$ is a ring satisfying
remark~\ref{assu:rdef}, then 
\[
C(X)\otimes\ring=\{\pz X\}
\]
 --- the Moore complex of $\pz X$ --- and we can define a chain map
\[
\gamma_{X}:C(\pz X)\otimes\ring\to C(X)\otimes\ring
\]
by $\ring$-linear extension. This chain-map induces a morphism of
Steenrod coalgebras:
\[
F_{X}:C(\pz X)\otimes\ring\to L_{\freeop}(C(X)\otimes\ring)
\]
(via the adjunction in equation~\ref{eq:cofree-adjunction} --- also
see diagram~\ref{eq:classifying-map}) where $L_{\freeop}(C(X)\otimes\ring)$
is the cofree coalgebra constructed in \cite{smith:cofree}.\end{defn}
\begin{prop}
\label{prop:composite-is-identity}If $X$ is a pointed, reduced simplicial
set with normalized chain-complex $N(X)$, $\ring$ is a ring satisfying
remark~\ref{assu:rdef}, and 
\[
h_{X}:X\to\pz X
\]
 is the Hurewicz map, (see definition~\ref{def:hurewicz-map}) then
the composite chain map
\[
C(X)\otimes\ring\xrightarrow{C(h_{X})}C(\pz X)\otimes\ring\xrightarrow{\gamma_{X}}C(X)\otimes\ring
\]
 is the identity map of $C(X)\otimes\ring$, where $\gamma_{X}$ is
defined in definition~\ref{def:pcmap}.\end{prop}
\begin{proof}
Just verify this on each simplex: if $x\in X$, then $h_{X}(x)=1\cdot x\in\pz X$
and $f(1\cdot x)=1\cdot x=x\in N(X)\otimes\ring$.\end{proof}
\begin{cor}
\label{cor:classifying-map-space-hurewicz}If $X$ is a pointed, reduced
degeneracy-free simplicial set, with normalized chain-complex $N(X)$,
$\ring$ is a ring satisfying remark~\ref{assu:rdef}, then the diagram
\[
\xymatrix{{N(X)\otimes\ring}\ar[rd]^{\alpha_{X}}\ar[d]_{N(h_{X})\otimes1} & {}\\
{C(\pz X)\otimes\ring}\ar[r]_{F_{X}} & {L_{\freeop}(C(X)\otimes\ring)}
}
\]
commutes, where
\begin{enumerate}
\item $h_{X}:X\to\pz X$ is the Hurewicz map (see definition~\ref{def:hurewicz-map}),
\item $\alpha_{X}:N(X)\otimes\ring\to L_{\freeop}(C(X)\otimes\ring)$ is
the unique morphism of Steenrod coalgebras induced by the chain-map
\[
\iota\otimes1:N(X)\otimes\ring\to C(X)\otimes\ring
\]
where $\iota$ is defined in lemma~\ref{lem:degeneracy-free-inclusion},
\item $F_{X}:C(\pz X)\otimes\ring\to L_{\freeop}(C(X)\otimes\ring)$ is
the unique morphism of Steenrod coalgebras induced by $\gamma_{X}$
in definition~\ref{def:pcmap}.
\end{enumerate}
\end{cor}
\begin{proof}
Since $h_{X}:X\to\pz X$ is \emph{simplicial,} $N(h_{X})\otimes1$
is a morphism of Steenrod coalgebras. Proposition~\ref{prop:composite-is-identity}
implies that the morphisms $\alpha_{X}$ and $F_{X}\circ(N(h_{X})\otimes1)$
are both induced by $\iota\otimes1$. Since induced maps to cofree
coalgebras are \emph{unique,} the triangle must commute (see theorem~\ref{thm:universal-steenrod}).
\end{proof}
One of the main results in this paper is:
\begin{thm}[Injectivity Theorem]
\label{thm:injectivity-theorem}Under the hypotheses of definition~\ref{def:pcmap},
the map 
\[
F_{X}:C(\pz X)\otimes\ring\to L_{\freeop}(C(X)\otimes\ring)
\]
 is injective. \end{thm}
\begin{rem*}
This is essentially the only place we need lemma~\ref{lem:degeneracy-free-inclusion},
which is the only reason we are interested in degeneracy-free simplicial
sets.

The commutative diagram in corollary~\ref{cor:classifying-map-space-hurewicz}
and this result imply that 
\[
N(h_{X})\otimes1=F_{X}^{-1}\circ\alpha_{X}:N(X)\otimes\ring\to C(\pz X)\otimes\ring
\]
so that the geometrically-relevant Hurewicz map is \emph{uniquely
determined} by the Steenrod coalgebra structure of $N(X)\otimes\ring$.\end{rem*}
\begin{proof}
See appendix~\ref{sec:Proof-of-theorem}.\end{proof}
\begin{prop}
\label{prop:hurewicz-commutes}If $X$ and $Y$ are pointed reduced
degeneracy-free simplicial sets, $\ring$ is a ring satisfying remark~\ref{assu:rdef},
and 
\[
f:N(X)\otimes\ring\to N(Y)\otimes\ring
\]
is a morphism of Steenrod coalgebras, then the diagram
\[
\xymatrix{{N(X)\otimes\ring}\ar[r]^{f}\ar[d]_{N(h_{X})\otimes1} & {N(Y)\otimes\ring}\ar[d]^{N(h_{Y})\otimes1}\\
{N(\pz X)\otimes\ring}\ar[r]_{N(\pgam f)\otimes1} & {N(\pz Y)\otimes\ring}
}
\]
 commutes, where $\pgam f:\pz X\to\pz Y$ is defined in proposition~\ref{prop:chain-maps-induce-zs}.\end{prop}
\begin{rem*}
The map $\pgam f:\pz X\to\pz Y$ is the main reason we are interested
in degeneracy-free simplicial sets:
\begin{quotation}
We would \emph{like} a map $\pz X\to\pz Y$ to exist with the property
that the induced map of \emph{Moore complexes}
\[
\moore{\pz X}\to\moore{\pz Y}
\]
\emph{coincides} with a given chain map $f:C(X)\otimes\ring\to C(Y)\otimes\ring$
of \emph{unnormalized chain-complexes. }Unfortunately we cannot guarantee
this unless we start with a chain-map of \emph{normalized} chain-complexes
and apply the $\pgam$-functor to it (see \ref{prop:chain-maps-induce-zs}).
Hence the need for normalized chain-complexes and degeneracy-free
simplicial sets.
\end{quotation}
These two data-points (i.e., the chain-complex and the chain-map induced
by the Hurewicz map) suffice to define $\ring^{\bullet}X$ --- the
cosimplicial space used to construct Bousfield and Kan's $\ring$-completion,
$\ring_{\infty}X$ (see \cite{bousfield-kan}).\end{rem*}
\begin{proof}
The fact that $\pgam f$ (see definition~\ref{prop:chain-maps-induce-zs})
maps each simplex of $\pz X$ (i.e., generator of $N(\pz X)$) via
$f$ implies that the diagram of \emph{chain-maps}
\[
\xyC{50pt}\xymatrix{{N(\pz X)\otimes\ring}\ar[r]^{C(\pgam f)\otimes1}\ar[d]_{\gamma_{X}} & {N(\pz Y)\otimes\ring}\ar[d]^{\gamma_{Y}}\\
{N(X)\otimes\ring}\ar[r]_{f} & {N(Y)\otimes\ring}
}
\]
commutes, where $\gamma_{X}$ and $\gamma_{Y}$ are given in definition~\ref{def:pcmap}.
The uniqueness of induced maps to cofree coalgebras (see definition~\ref{def:cofreecoalgebra}
and theorem~\ref{thm:universal-steenrod}) and the fact that the
target, $L_{\freeop}\left(C(Y)\otimes\ring\right)$, is \emph{cofree}
implies that the induced diagram of Steenrod coalgebras
\[
\xymatrix{{C(\pz X)\otimes\ring}\ar[r]^{C(\pgam f)\otimes1}\ar[d]_{F_{X}} & {C(\pz Y)\otimes\ring}\ar[d]^{F_{Y}}\\
{L_{\freeop}\left(C(X)\otimes\ring\right)}\ar[r]_{L_{\freeop}f} & {L_{\freeop}\left(C(Y)\otimes\ring\right)}
}
\]
commutes. The conclusion follows from the commutativity of the diagram
\[
\xymatrix{{N(\pz X)\otimes\ring}\ar[r]^{N(\pgam f)\otimes1} & {N(\pz Y)\otimes\ring}\\
{C(\pz X)\otimes\ring}\ar[r]^{C(\pgam f)\otimes1}\ar[d]_{F_{X}}\ar[u] & {C(\pz Y)\otimes\ring}\ar[d]^{F_{Y}}\ar[u]\\
{L_{\freeop}\left(C(X)\otimes\ring\right)}\ar[r]_{L_{\freeop}\{\pgam f\}\otimes1} & {L_{\freeop}\left(C(Y)\otimes\ring\right)}\\
{N(X)\otimes\ring}\ar[r]_{f}\ar[u]^{\alpha_{X}}\ar@/^{4pc}/[uu]^{N(h_{X})\otimes1} & {N(Y)\otimes\ring}\ar[u]_{\alpha_{Y}}\ar@/_{4pc}/[uu]_{N(h_{Y})\otimes1}
}
\]
where $\alpha_{X}:N(X)\otimes\ring\to L_{\freeop}\left(C(X)\otimes\ring\right)$
and $\alpha_{Y}:N(Y)\otimes\ring\to L_{\freeop}\left(C(Y)\otimes\ring\right)$
are induced by the inclusions of Steenrod coalgebras
\begin{align*}
\iota_{X}:N(X) & \to C(X)\\
\iota_{Y}:N(Y) & \to C(Y)
\end{align*}
 respectively (compare corollary~\ref{cor:main-hurewicz}).
\end{proof}
Our main topological result is
\begin{thm}
\label{cor:main-hurewicz}Let $X$ and $Y$ be pointed, reduced degeneracy-free
simplicial sets with normalized chain-complexes $N(X)$ and $N(Y)$,
respectively, with their functorial Steenrod diagonals. If $\ring$
is a ring satisfying remark~\ref{assu:rdef} and 
\[
f:N(X)\otimes\ring\to N(Y)\otimes\ring
\]
is a morphism of Steenrod coalgebras, then $f$ induces
\[
f_{\infty}:\ring_{\infty}X\to\ring_{\infty}Y
\]
of $\ring$-completions that makes the diagram 
\begin{equation}
\xymatrix{{X}\ar[d]_{\phi_{X}} & {Y}\ar[d]^{\phi_{Y}}\\
{\ring_{\infty}X}\ar[r]^{f_{\infty}}\ar[d]_{q_{X}} & {\ring_{\infty}Y}\ar[d]^{q_{Y}}\\
{\pz X}\ar[r]_{\tilde{f}} & {\pz Y}
}
\label{eq:hurewicz-main-dia}
\end{equation}
commute. Here\textup{ }
\begin{eqnarray*}
\phi_{X}:X & \to & \ring_{\infty}X\\
\phi_{Y}:Y & \to & \ring_{\infty}Y
\end{eqnarray*}
 are the canonical maps (see 4.2 in \cite[chapter~I]{bousfield-kan})
and $q_{X}$ and $q_{Y}$ are maps to the 0-coskeleta.\end{thm}
\begin{rem*}
The vertical composites are just the respective\emph{ Hurewicz maps.}\end{rem*}
\begin{proof}
Proposition~\ref{prop:chain-maps-induce-zs} implies that the chain-map,
$f$ induces morphisms of simplicial abelian groups
\begin{equation}
\pz^{i-1}\pgam f:\pz^{i}X\to\pz^{i}Y\label{eq:pz-i}
\end{equation}
 for all $i>0$. The fact that $f$ preserves Steenrod diagonals and
proposition~\ref{prop:hurewicz-commutes} implies that the diagram
\begin{equation}
\xymatrix{{N(X)\otimes\ring}\ar[r]^{f}\ar[d]_{N(h_{X})} & {N(Y)\otimes\ring}\ar[d]^{N(h_{Y})\otimes1}\\
{N(\pz X)\otimes\ring}\ar[r]_{N(\{\pgam f\})} & {N(\pz Y)\otimes\ring}
}
\label{eq:initial-hurewicz-commute}
\end{equation}
 commutes, where $h_{X}$ and $h_{Y}$ are Hurewicz maps. If we take
$\pgam*$ of this diagram (\ref{eq:initial-hurewicz-commute}), proposition~\ref{prop:dold-kan}
implies that we get a commutative diagram of simplicial abelian groups
\begin{equation}
\xymatrix{{\pz X}\ar[r]^{\pgam f}\ar[d]_{\pz h_{X}} & {\pz Y}\ar[d]^{\pz h_{Y}}\\
{\pz^{2}X}\ar[r]_{\pz\pgam f} & {\pz^{2}Y}
}
\label{eq:hurewicz2-commute}
\end{equation}
 Now recall the cosimplicial resolutions $\ring^{\bullet}X$ and $\ring^{\bullet}Y$
defined in example~4.1 of \cite[chapter~VII, section~4]{goerss-jardine}.
They have levels 
\[
(\ring^{\bullet}X)^{n}=\pz^{n+1}X
\]
$n\ge0$, with coface maps
\begin{align*}
\delta_{X}^{i}=\pz^{i}h\pz^{n-i+1}: & \pz^{n+1}X\to\pz^{n+2}X\\
\delta_{Y}^{i}=\pz^{i}h\pz^{n-i+1}: & \pz^{n+1}Y\to\pz^{n+2}Y
\end{align*}
for $i=0,\dots n+1$, where $h:*\to\pz*$ is the Hurewicz map of the
space to its right. In addition, they have codegeneracy maps
\begin{align*}
s_{X}^{i}=\pz^{i}\gamma\pz^{n-i} & :\pz^{n+2}X\to\pz^{n+1}X\\
s_{Y}^{i}=\pz^{i}\gamma\pz^{n-i} & :\pz^{n+2}Y\to\pz^{n+1}Y
\end{align*}
for $i=0,\dots,n$, where $\gamma(\alpha\cdot\beta)=\alpha\beta$
for $\alpha,\beta\in\ring$. If $0\le i<n+1$ the diagram\foreignlanguage{english}{{]}\textasciicircum{}\{F\_\{Y\}\}}
\[
\xyC{50pt}\xymatrix{{\pz^{n-i+1}X}\ar[r]^{\pz^{n-i}\pgam f}\ar[d]_{h_{X}} & {\pz^{n-i+1}Y}\ar[d]^{h_{Y}}\\
{\pz^{n-i+2}X}\ar[r]_{\pz^{n-i+1}\pgam f} & {\pz^{n-i+2}Y}
}
\]
 commutes by the naturality of Hurewicz maps. Composing this with
$\pz^{i}$ shows that the maps $\pz^{n}\pgam f$ preserve cofaces
$\delta_{*}^{i}$ for $i<n+1$. The \emph{only} coface that uses the
topology of $X$ and $Y$ --- beyond their bare chain-complexes ---
is (remarkably!) $\delta_{*}^{n+1}$. Applying $\pz^{n}$ to diagram~\ref{eq:hurewicz2-commute}
implies that this is \emph{also} preserved. It follows that the maps
in equation~\ref{eq:pz-i} commute with all cofaces and codegeneracies
so that they define a morphism of cosimplicial spaces
\[
\ring^{\bullet}f:\ring^{\bullet}X\to\ring^{\bullet}Y
\]
that induces a morphism $f_{\infty}$ of total spaces that makes diagram~\ref{eq:hurewicz-main-dia}
commute.\end{proof}
\begin{prop}
Under the hypotheses of corollary~\ref{cor:main-hurewicz}, if $f$
is surjective, $\ring_{\infty}f$ is a fibration. If $f$ is a surjective
homology equivalence, then $\ring_{\infty}f$ is a trivial fibration. \end{prop}
\begin{proof}
All of the coface maps except for the $0^{\text{th}}$ in $\ints^{\bullet}X$
are morphisms of simplicial abelian groups. It follows that $\ring^{\bullet}X$
is ``group-like'' in the sense of section~4 in chapter~X of \cite{bousfield-kan}.
The conclusion follows from proposition~4.9 section~4 in chapter~X
of \cite{bousfield-kan}. 

If $f$ is also a homology equivalence, then $\ring^{\bullet}f:\ring^{\bullet}X\to\ring^{\bullet}Y$
is a pointwise trivial fibration. The final statement follows from
theorem~4.13 in chapter~VIII of \cite{bousfield-kan}, and the fact
that $\Delta^{\bullet}$ is cofibrant in $\coS$.
\end{proof}
If $\ring=\ints$ and spaces are \emph{nilpotent,} we can say a bit
more:
\begin{cor}
\label{cor:cfxy-y-nilpotent}Under the hypotheses of corollary~\ref{cor:main-hurewicz},
if $Y$ is also nilpotent and a Kan complex then $\phi_{Y}:Y\to\ints_{\infty}Y$
is a weak equivalence with a homotopy-inverse, $\phi':\ints_{\infty}Y\to Y$,
that fits into a commutative diagram 
\begin{equation}
\xymatrix{{X}\ar[d]_{\phi_{X}} & {Y}\ar[d]_{\phi_{Y}}\\
{\ints_{\infty}X}\ar[r]_{\ints_{\infty}f} & {\ints_{\infty}Y}\ar@/_{1pc}/[u]_{\phi'}
}
\label{eq:cfxy-y-nilpotent}
\end{equation}
where
\begin{enumerate}
\item $\phi_{Y}:Y\to\ints_{\infty}X$ is a weak equivalence
\item a morphism of cellular coalgebras, $f:\cf X\to\cf Y$, induces a map
of simplicial sets 
\[
X\xrightarrow{\phi_{X}}\ints_{\infty}X\xrightarrow{\ints_{\infty}f}\ints Y\xrightarrow{\phi'}Y
\]

\end{enumerate}
If $Y$ is not a Kan complex, then the diagram that results from applying
the topological realization functor, $|*|$, to all terms of diagram~\ref{eq:cfxy-y-nilpotent},
commutes.\end{cor}
\begin{rem*}
For instance, singular simplicial sets are always Kan complexes.\end{rem*}
\begin{proof}
The main statement (that $\phi_{Y}$ is a weak equivalence) follows
from proposition~3.5 in chapter~V of \cite{bousfield-kan}.
\end{proof}
Our final result is:
\begin{cor}
\label{cor:final-result}If $X$ and $Y$ are pointed, reduced, nilpotent
degeneracy-free simplicial sets that are Kan complexes, with normalized
chain-complexes $N(X)$ and $N(Y)$, respectively, then $X$ and $Y$
are homotopy equivalent if and only if there exists a morphism of
Steenrod coalgebras 
\[
f:N(X)\to N(Y)
\]
inducing isomorphisms in homology over $\ints$. If $X$ and $Y$
are not Kan complexes, the corresponding statement holds for their
topological realizations.\end{cor}
\begin{proof}
Any homotopy equivalence $g:X\to Y$ induces a map like $f$ in the
statement. Conversely, given $f$ as above, we get
\begin{equation}
\xymatrix{{X}\ar[d]_{\phi_{X}} & {Y}\ar[d]^{\phi_{Y}}\\
{\ints_{\infty}X}\ar[r]^{f_{\infty}}\ar[d]_{q_{X}} & {\ints_{\infty}Y}\ar[d]^{q_{Y}}\\
{\tilde{\ints}X}\ar[r]_{\pgam f} & {\tilde{\ints}Y}
}
\label{eq:hurewicz-main-dia-1}
\end{equation}
where $\phi_{X}$ and $\phi_{Y}$ are weak equivalences and the map
at the bottom is a weak equivalence. Since $\phi_{X}$ and $\phi_{Y}$
are weak equivalences, it follows that the maps $q_{X}$ and $q_{Y}$
are homotopic to the Hurewicz maps of $\ints_{\infty}X$ and $\ints_{\infty}Y$,
respectively. Since $f$ is a weak equivalence, it follows that $f_{\infty}$
induces isomorphisms in homology (recall that $\pi_{n}(\tilde{\ints}X)\cong H_{n}(X)=H_{n}(\ints_{\infty}X)$
for all $n\ge0$, and that a corresponding statement holds for $Y$).
Whitehead's theorem implies the existence of a homotopy inverse for
$\phi_{Y}$ and 
\[
\phi_{Y}^{-1}\circ f_{\infty}\circ\phi_{X}:X\to Y
\]
 is a homotopy equivalence.
\end{proof}
Since arbitrary simplicial sets are homotopy equivalent to degeneracy-free
ones, we also get
\begin{cor}
\label{cor:statement-for-simplicial-sets}If $X$ and $Y$ are pointed
reduced simplicial sets, $\ring$ is a ring satisfying remark~\ref{assu:rdef},
then any morphism of Steenrod coalgebras (over unnormalized chain-complexes)
\[
f:C(X)\otimes\ring\to C(Y)\otimes\ring
\]
 induces a commutative diagram
\[
\xymatrix{{X} & {Y}\\
\ds\circ\sd(X)\ar[d]_{\phi_{(\ds\circ\sd(X))}}\ar[u]^{g_{X}} & \ds\circ\sd(Y)\ar[d]^{\phi_{(\ds\circ\sd(Y))}}\ar[u]_{g_{Y}}\\
{\ring_{\infty}(\ds\circ\sd(X))}\ar[r]^{f_{\infty}}\ar[d]_{q_{(\ds\circ\sd(X))}} & {\ring_{\infty}(\ds\circ\sd(Y))}\ar[d]^{q_{(\ds\circ\sd(Y))}}\\
{\pz(\ds\circ\sd(X))}\ar[r]_{\pgam f} & {\pz(\ds\circ\sd(Y))}
}
\]
where;
\begin{enumerate}
\item the functors $\sd$ and $\ds$ are defined in definition~\ref{def:sd-ds-functors}
\item the maps $g_{X}$ and $g_{Y}$ are defined in equation~\ref{eq:ds-sd-unit}
and are homotopy equivalences, by proposition~\ref{prop:homotopy-equiv-sd-ds}.
\end{enumerate}
In particular, if $X$ and $Y$ are nilpotent, $\ring=\ints$, and
$f$ is a homology equivalence, then the topological realizations,
$|X|$ and $|Y|$, are homotopy equivalent.\end{cor}
\begin{proof}
This follows immediately from corollary~\ref{cor:main-hurewicz}
and proposition~\ref{prop:cx-is-nfx}, which implies that $N(\ds\circ\sd(X))=C(X)$.
\end{proof}
\appendix

\section{Simplicial sets and Delta-complexes\label{sec:Simplicial-sets-and}}

We recount results of Rourke and Sanderson (see\cite{rourke-sanderson-delta-complex})
involving variations on the concept of simplicial set.
\begin{defn}
\label{def:delta-complexes}Let $\ords$ be the ordinal number category
whose morphisms are order-preserving monomorphisms between them. The
objects of $\ords$ are elements $\mathbf{n}=\{0\to1\to\cdots\to n\}$
and a morphism 
\[
\theta:\mathbf{m}\to\mathbf{n}
\]
 is a strict order-preserving map ($i<k\implies\theta(i)<\theta(j)$).
Then the category of \emph{delta-complexes,} $\dcat$, has objects
that are contravariant functors
\[
\ords\to\mathbf{Set}
\]
to the category of sets. The chain complex of a delta-complex, $X$,
will be denoted $N(X)$.\end{defn}
\begin{rem*}
In other words, delta-complexes are just simplicial sets \emph{without
degeneracy-operators.} The topological realization of a delta-complex
is the result of gluing together simplices via the face-maps.

A simplicial set gives rise to a delta-complex by ``forgetting''
its degeneracies --- ``promoting'' its degenerate simplices to nondegenerate
status. The topological realization of the resulting delta-complex
is vastly ``larger'' than that of the original simplicial set. 

Conversely, a delta-complex can be converted into a simplicial set
by equipping it with degenerate simplices in a mechanical fashion.
These operations define functors:\end{rem*}
\begin{defn}
\label{def:sd-ds-functors}The functor
\[
\sd:\ss\to\dcat
\]
is defined to simply drop degeneracy operators (degenerate simplices
become nondegenerate) while retaining the face-operators. The functor
\[
\ds:\dcat\to\ss
\]
equips a delta complex, $X$, with degenerate simplicies and operators
via
\begin{equation}
\ds(X)_{m}=\bigsqcup_{\mathbf{m}\twoheadrightarrow\mathbf{n}}X_{n}\label{eq:ds-functor}
\end{equation}
for all $m>n\ge0$.\end{defn}
\begin{rem*}
The functors $\sd$ and $\ds$ were denoted $F$ and $G$, respectively,
in \cite{rourke-sanderson-delta-complex}. Equation~\ref{eq:ds-functor}
simply states that we add all possible degeneracies of simplices in
$X$ subject \emph{only} to the basic identities that face- and degeneracy-operators
must satisfy. 

Although $\sd$ promotes degenerate simplicies to nondegenerate ones,
Rourke and Sanderson's paper, \cite{rourke-sanderson-delta-complex},
shows that these new nondegenerate simplices can be collapsed without
changing the homotopy type of the complex: although the degeneracy
operators are no longer built in to the delta-complex, they still
define contracting homotopies.
\end{rem*}
The definition immediately implies that
\begin{prop}
\label{prop:cx-is-nfx}If $X$ is a simplicial set and $Y$ is a delta-complex,
$C(X)=N(\sd(X))$, $N(\ds(Y))=N(Y)$, and $C(X)=N(\ds\circ\sd(X))$.\end{prop}
\begin{defn}
\label{def:degeneracy-free}A simplicial set, $X$, is defined to
be \emph{degeneracy-free} if 
\[
X=\ds(Y)
\]
 for some delta-complex, $Y$.\end{defn}
\begin{rem*}
Compare definition~1.10 in chapter VII of \cite{goerss-jardine}\footnote{ Their definition has a typo, stating that $\ords$ consists of \emph{surjections}
rather than \emph{injections}.}). In a manner of speaking, $X$ is freely generated by the degeneracy
operators acting on a basis consisting of the simplices of $Y$. Lemma~1.2
in chapter~VII of \cite{goerss-jardine} describes other properties
of degeneracy-free simplicial sets (hence of the functor $\ds$).
\end{rem*}
In \cite{rourke-sanderson-delta-complex}, Rourke and Sanderson also
showed that one could give a ``somewhat more intrinsic'' definition
of degeneracy-freeness:
\begin{prop}
\label{prop:intrinsic-degenracy-free}If $X$ is a simplicial set,
let $\mathrm{Core}(X)$ consist of the nondegenerate simplices and
their faces. This is a delta-complex and there exists a canonical
map
\[
c:\ds(\mathrm{Core}(X))\to X
\]
sending simplices of $\mathrm{Core}(X)$ to themselves in $X$ and
degeneracies to suitable degeneracies of them. Then $X$ is degeneracy-free
if and only if $c$ is an isomorphism.
\end{prop}
Theorem~1.7 of \cite{rourke-sanderson-delta-complex} shows that
there exists an adjunction:

\begin{equation}
\ds:\dcat\leftrightarrow\ss:\sd\label{eq:ds-sd-adjunction}
\end{equation}
The composite (the \emph{counit} of the adjunction)
\[
\sd\circ\ds:\dcat\to\dcat
\]
maps a delta complex into a much larger one --- that has an infinite
number of (degenerate) simplices added to it. There is a natural inclusion
\[
\iota:X\to\sd\circ\ds(X)
\]
 and a natural map (the \emph{unit} of the adjunction)
\begin{equation}
g:\ds\circ\sd(X)\to X\label{eq:ds-sd-unit}
\end{equation}
The functor $g$ sends degenerate simplices of $X$ that had been
``promoted to nondegenerate status'' by $\sd$ to their degenerate
originals --- and the extra degenerates added by $\ds$ to suitable
degeneracies of the simplices of $X$. 

Rourke and Sanderson also prove: 
\begin{prop}
\label{prop:homotopy-equiv-sd-ds}If $X$ is a simplicial set and
$Y$ is a delta-complex then
\begin{enumerate}
\item $|Y|$ and $|\ds Y|$ are homeomorphic
\item the map $|g|:|\ds\circ\sd(X)|\to|X|$ is a homotopy equivalence. 
\item $\sd:H\ss\to H\dcat$ defines an equivalence of categories, where
$H\ss$ and $H\dcat$ are the homotopy categories, respectively, of
$\ss$ and $\dcat$. The inverse is $\ds:H\dcat\to H\ss$. In particular,
if $X$ is a Kan complex, the natural map
\[
g:\ds\circ\sd(X)\to X
\]
is a homotopy equivalence.
\end{enumerate}
\end{prop}
\begin{rem*}
Here, $|*|$ denotes the topological realization functors for $\ss$
and $\dcat$.\end{rem*}
\begin{proof}
The first two statements are proposition~2.1 of \cite{rourke-sanderson-delta-complex}
and statement~3 is theorem~6.9 of the same paper. The final statement
follows from Whitehead's theorem and the fact that Kan complexes are
fibrant in the Quillen model structure of simplicial sets (see \cite{goerss-jardine}).
\end{proof}

\section{Functorial Steenrod diagonals\label{sec:Functorial-Steenrod-coalgebras}}

In this section, we construct a functorial Steenrod coalgebra structure
described in proposition~\ref{prop:cf-intro}. Also see \cite{rocio-steenrod}
for an alternative functorial form of Steenrod coalgebra.

The Steenrod diagonal was orginally developed by Steenrod in \cite{steenrod-cup-i},
and that paper's result, 12.4, is a dual to our main result (proposition~\ref{pro:simplicespropertyS}).
That paper's emphsis was completely different from that of the present
paper and it didn't use the concept of simplicial sets. This appendix
develops the Steenrod diagonal in a way that is clearly functorial
and uses modern notation.

We begin with a contracting cochain on the normalized chain-complex
of a standard simplex:
\begin{defn}
\label{def:simplex-contracting-cochain}Let $\Delta^{k}$ be a standard
$k$-simplex with vertices $\{[0],\dots,[k]\}$ and $j$-faces $\{[i_{0},\dots,i_{j}]\}$
with $i_{0}<\cdots<i_{j}$ and let $s^{k}$ denote its normalized
chain-complex with boundary map $\partial$. This is equipped with
an augmentation
\[
\epsilon:s^{k}\to\ints
\]
that maps all vertices to $1\in\ints$ and all other simplices to
$0$. Let 
\[
\iota_{k}:\ints\to s^{k}
\]
 denote the map sending $1\in\ints$ to the image of the vertex $[k]$.
Then we have a contracting cochain\textit{\emph{
\begin{equation}
\varphi_{k}([i_{0},\dots,i_{t}]=\left\{ \begin{array}{cc}
(-1)^{t+1}[i_{0},\dots,i_{t},k] & \mathrm{if}\,i_{t}\ne k\\
0 & \mathrm{if}\,i_{t}=k
\end{array}\right.\label{eq:simplex-contracting-cochain}
\end{equation}
and $1-\iota_{k}\circ\epsilon=\partial\circ\varphi_{k}+\varphi_{k}\circ\partial$.}}\end{defn}
\begin{thm}
\label{thm:ns-construct}The unnormalized chain-complex, $U^{k}$,
of $[i_{0},\dots,i_{k}]=\Delta^{k}$ has a Steenrod coalgebra structure
\[
\xi:\rs 2\otimes U^{k}\to U^{k}\otimes U^{k}
\]
that is natural with respect to order-preserving mappings of vertex-sets
\[
[i_{0},\dots,i_{k}]\to[j_{0},\dots,j_{\ell}]
\]
with $j_{0}\le\cdots\le j_{\ell}$ and $\ell\ge k$. If $c\in U^{k}$
is degenerate then one of the two factors in each term of $h(*\otimes c)$
is degenerate so that $h$ induces a well-defined Steenrod coalgebra
structure on the normalized chain-complex of $\Delta^{k}$, denoted
$\ns k$.\end{thm}
\begin{rem*}
On the author's web site (under the Research link), there is a Python
program for computing $\xi(x\otimes C(\Delta^{k}))$. As in \cite{rocio-steenrod},
the number of terms grows exponentially with $k$ and the dimension
of $x\in\rs 2$.\end{rem*}
\begin{proof}
\textit{\emph{If $C=s^{k}=C(\Delta^{k})$ --- the }}\textit{normalized}\textit{\emph{
chain complex --- we can define a corresponding contracting homotopy
on $C\otimes C$ via
\[
\Phi=\varphi_{k}\otimes1+\iota_{k}\circ\epsilon\otimes\varphi_{k}
\]
where $\varphi_{k}$, $\iota_{k}$, and $\epsilon$ are as in definition~\ref{def:simplex-contracting-cochain}.
}}Above dimension $0$, $\Phi$ is effectively equal to $\varphi_{k}\otimes1$\textit{\emph{.
Now set $M_{2}=C\otimes C$ and $N_{2}=\img(\Phi)$. In dimension
$0$, we define $f_{2}$ for all $n$ via:
\[
\xi(A\otimes[0])=\left\{ \begin{array}{ll}
[0]\otimes[0] & \mathrm{if}\,A=[\,]\\
0 & \mathrm{if}\,\dim A>0
\end{array}\right.
\]
This clearly makes $s^{0}$ a Steenrod coalgebra.}}

\textit{\emph{Suppose that the $\xi$ are defined below dimension
$k$. Then the Steenrod coalgebra structure of $C(\partial\Delta^{k})$
is well-defined and satisfies the conclusions of this theorem. We
define $\xi(a[a_{1}|\dots|a_{j}]\otimes[0,\dots,k])$ by induction
on $j$, }}
\begin{eqnarray}
\xi(A\otimes s^{k}) & = & \Phi\circ\xi(\partial A\otimes s^{k})\nonumber \\
 & + & (-1)^{\dim A}\Phi\circ\xi(A\otimes\partial s^{k})\label{eq:high-diag-comp}
\end{eqnarray}
 where $A\in A(\ints_{2},1)\subset\rs n$ and the term $\xi(A\otimes\partial s^{k})$
refers to the Steenrod coalgebra structure of $C(\partial\Delta^{k})$.

The terms $\xi(A\otimes\partial s^{k})$ and $\xi(\partial A\otimes s^{k})$
are defined by induction on the dimension of $A$ and we ultimately
get an expression for $\xi(x\otimes[0,\dots,k])$ as a sum of tensor-products
of sub-simplices of $[0,\dots,k]$ --- given as ordered lists of vertices.

We claim that this Steenrod coalgebra structure is natural with respect
to ordered mappings of vertices. This follows from the fact that the
only significant property that the vertex $k$ \emph{has} in equations~\ref{eq:simplex-contracting-cochain}
and \ref{eq:high-diag-comp} is that it is the \emph{highest numbered}
vertex. 

The final statement (regarding degenerate simplices) follows from
three facts:
\begin{enumerate}
\item It is true for the \emph{Alexander-Whitney} coproduct (the starting
point of our induction),
\item The boundary of a degenerate simplex is a linear combination of degenerate
simplices, and
\item $\Phi$ of a term with a degenerate factor has a degenerate factor.
\end{enumerate}
\end{proof}
Here is an example of some higher coproducts:
\begin{example}
\label{example:e1timesdelta2}If $[0,1,2]=\Delta^{2}$ is a $2$-simplex,
then

\begin{equation}
\xi([\,]\otimes\Delta^{2})=\Delta^{2}\otimes F_{0}F_{1}\Delta^{2}+F_{2}\Delta^{2}\otimes F_{0}\Delta^{2}+F_{1}F_{2}\Delta^{2}\otimes\Delta^{2}\label{eq:delta-2-coproduct}
\end{equation}
--- the standard (Alexander-Whitney) coproduct --- and

\begin{align*}
\xi([(1,2)]\otimes\Delta^{2})= & [0,1,2]\otimes[1,2]-[0,2]\otimes[0,1,2]\\
 & -[0,1,2]\otimes[0,1]
\end{align*}
or, in face-operations

\begin{align}
\xi([(1,2)]\otimes\Delta^{2})= & \Delta^{2}\otimes F_{0}\Delta^{2}-F_{1}\Delta^{2}\otimes\Delta^{2}\label{eq:e1timesdelta2}\\
 & -\Delta^{2}\otimes F_{2}\Delta^{2}\nonumber 
\end{align}
\end{example}
\begin{proof}
If we write $\Delta^{2}=[0,1,2]$, we get
\[
\xi([\,]\otimes\Delta^{2})=[0,1,2]\otimes[2]+[0,1]\otimes[1,2]+[0]\otimes[0,1,2]
\]

To compute $\xi([(1,2)]_{2}\otimes\Delta^{2})$ we have a version
of equation~\ref{eq:high-diag-comp}:
\begin{align*}
\xi(e_{1}\otimes\Delta^{2}) & =\Phi_{2}(\xi(\partial e_{1}\otimes\Delta^{2})-\Phi_{2}\xi(e_{1}\otimes\partial\Delta^{2})\\
 & =-\Phi_{2}(\xi((1,2)\cdot[\,]\otimes\Delta^{2})+\Phi_{2}(\xi([\,]\otimes\Delta^{2})-\Phi_{2}\xi(e_{1}\otimes\partial\Delta^{2})
\end{align*}
Now 
\begin{align*}
\Phi_{2}(1,2)\cdot(\xi([\,]\otimes\Delta^{2})= & (\varphi_{2}\otimes1)\bigl([2]\otimes[0,1,2]-[1,2]\otimes[0,1]\\
 & +[0,1,2]\otimes[0]\bigr)\\
 & +(i\circ\epsilon\otimes\varphi_{2})\bigl([2]\otimes[0,1,2]\\
 & -[1,2]\otimes[0,1]+[0,1,2]\otimes[0]\bigr)\\
= & 0
\end{align*}
and
\begin{align*}
\Phi_{2}\xi([\,]\otimes\Delta^{2})= & (\varphi_{2}\otimes1)\bigl([0,1,2]\otimes[2]+[0,1]\otimes[1,2]\\
 & +[0]\otimes[0,1,2]\bigr)\\
= & [0,1,2]\otimes[1,2]-[0,2]\otimes[0,1,2]
\end{align*}
In addition, proposition~\ref{pro:simplicespropertyS} implies that
\begin{align*}
\xi(e_{1}\otimes\partial\Delta^{2})= & [1,2]\otimes[1,2]-[0,2]\otimes[0,2]\\
 & +[0,1]\otimes[0,1]
\end{align*}
so that
\[
\Phi_{2}\xi(e_{1}\otimes\partial\Delta^{2})=[0,1,2]\otimes[0,1]
\]

We conclude that
\begin{align*}
f_{2}([(1,2)]\otimes\Delta^{2})= & [0,1,2]\otimes[1,2]-[0,2]\otimes[0,1,2]\\
 & -[0,1,2]\otimes[0,1]
\end{align*}
which implies equation~\ref{eq:e1timesdelta2}.
\end{proof}
We can extend the Steenrod coalgebra structure on simplices to one
on degenerate simplices by regarding
\[
D_{i}\Delta^{n}=D_{i}[0,\dots,n]=[0,\dots,i,i,\dots n]
\]
and plugging these vertices into the formulas for the higher coproducts.
For instance, example~\ref{example:e1timesdelta2} implies that
\begin{align*}
\xi([(1,2)]\otimes D_{0}[0,1])= & [0,0,1]\otimes[0,1]-[0,1]\otimes[0,0,1]\\
 & -[0,0,1]\otimes[0,0]
\end{align*}
or 
\begin{align*}
\xi([(1,2)]\otimes D_{0}\Delta^{1})= & D_{0}\Delta^{1}\otimes\Delta^{1}-\Delta^{1}\otimes D_{0}\Delta^{1}\\
 & -D_{0}\Delta^{1}\otimes D_{0}F_{1}\Delta^{1}
\end{align*}
and 
\begin{align*}
\xi([(1,2)]\otimes D_{1}[0,1])= & [0,1,1]\otimes[1,1]-[0,1]\otimes[0,1,1]\\
 & -[0,1,1]\otimes[1,1]
\end{align*}
or
\begin{align*}
\xi([(1,2)]\otimes D_{1}\Delta^{1})= & D_{1}\Delta^{1}\otimes D_{0}F_{0}\Delta^{1}-\Delta^{1}\otimes D_{1}\Delta^{1}\\
 & -D_{1}\Delta^{1}\otimes D_{0}F_{0}\Delta^{1}
\end{align*}

It follows that:
\begin{prop}
\emph{\label{prop:cf-functor}If $X$ is a simplicial set, we can
define a natural Steenrod coalgebra structure on the unnormalized
chain-complex of $X$:}
\[
C(X)=\lim_{\to}\ns k=\dlimit C(\Delta^{k})
\]
for $\Delta^{n}\in\boldsymbol{\Delta}\downarrow X$ --- the simplex
category of $X$ with Steenrod diagonal
\[
\xi:\cf X\to\cf X\otimes\cf X
\]
This induces a natural Steenrod coalgebra structure on the normalized
chain-complex
\[
\xi:N(X)\to N(X)\otimes N(X)
\]
\end{prop}
\begin{proof}
Theorem~B.3 of \cite{smith:model-cats} implies that this colimit
of Steenrod coalgebras (i.e., coalgebras over $\freeop$) has a chain-complex
that is the chain-complex of the colimit of \emph{chain-complexes}
--- i.e. the unnormalized chain complex of $X$.

The second statement follows from the fact that $\xi$ of a degenerate
simplex has a degenerate factor in $\cf X\otimes\cf X$.
\end{proof}
We conclude this section with a calculation that is crucial to this
paper:
\begin{prop}
\label{pro:simplicespropertyS}Let $X$ be a simplicial set with $C=C(X)$
and with Steenrod coalgebra structure 
\[
\xi:\rs 2\otimes C(X)\to C(X)\otimes C(X)
\]
and suppose $\rs 2$ is generated in dimension $n$ by $e_{n}=\underbrace{[(1,2)|\cdots|(1,2)]}_{n\text{ terms}}$.
If $x\in C$ is the image of a $k$-simplex, then
\[
\xi(e_{k}\otimes x)=\eta_{k}\cdot x\otimes x
\]
where $\eta_{k}=(-1)^{k(k-1)/2}$.\end{prop}
\begin{rem*}
This is just a chain-level statement that the Steenrod operation $\operatorname{Sq}^{0}$
acts trivially on mod-$2$ cohomology. Compare this to 12.4 in \cite{steenrod-cup-i}.
A weaker form of this result appeared in \cite{davis:mco}.

It proves that Steenrod coalgebras of the form $C(X)$, for a simplicial
set $X$ are \emph{not nilpotent:} iterated coproducts of \emph{simplices}
never ``peter out''. This turns out to provide a way to ``recognize''
simplices among the elements of $C(X)$.\end{rem*}
\begin{proof}
Recall that $(\rs 2)_{n}=\ints[\ints_{2}]$ generated by $e_{n}=[\underbrace{(1,2)|\cdots|(1,2)}_{n\text{ factors}}]$.
Let $T$ be the generator of $\ints_{2}$ --- acting on $C\otimes C$
by swapping the copies of $C$.

We assume that $\xi(e_{i}\otimes C(\Delta^{j}))\subset C(\Delta^{j})\otimes C(\Delta^{j})$
so that 
\begin{equation}
i>j\implies\xi(e_{i}\otimes C(\Delta^{j}))=0\label{eq:big-diag-condition}
\end{equation}
\textit{\emph{We extend this to }}$C\otimes C$ via
\[
\Phi_{k}=\varphi_{k}\otimes1+\iota_{k}\circ\epsilon\otimes\varphi_{k}
\]
and use the Koszul convention on signs. Note that $\Phi_{k}^{2}=0$.
As in section~4 of \cite{smith:1994}, if $e_{0}\in\rs 2$ is the
$0$-dimensional generator, we define
\[
\xi:\rs 2\otimes C\to C\otimes C
\]
 inductively by
\begin{eqnarray}
\xi(e_{0}\otimes[i]) & = & [i]\otimes[i]\nonumber \\
\xi(e_{0}\otimes[0,\dots,k]) & = & \sum_{i=0}^{k}[0,\dots,i]\otimes[i,\dots,k]\label{eq:big-diag1}
\end{eqnarray}
Let $\sigma=\Delta^{k}$ and inductively define
\begin{align*}
\xi(e_{k}\otimes\sigma) & =\Phi_{k}(\xi(\partial e_{k}\otimes\sigma)+(-1)^{k}\Phi_{k}\xi(e_{k}\otimes\partial\sigma)\\
 & =\Phi_{k}(\xi(\partial e_{k}\otimes\sigma)
\end{align*}
because of equation~\ref{eq:big-diag-condition}. 

Expanding $\Phi_{k}$, we get
\begin{align}
\xi(e_{k}\otimes\sigma) & =(\varphi_{k}\otimes1)(\xi(\partial e_{k}\otimes\sigma))+(i\circ\epsilon\otimes\varphi_{k})\xi(\partial e_{k}\otimes\sigma)\nonumber \\
 & =(\varphi_{k}\otimes1)(\xi(\partial e_{k}\otimes\sigma))\label{eq:big-diag2}
\end{align}
 because $\varphi_{k}^{2}=0$ and $\varphi_{k}\circ i\circ\epsilon=0$. 

Noting that $\partial e_{k}=(1+(-1)^{k}T)e_{k-1}\in\rs 2$, we get
\begin{align*}
\xi(e_{k}\otimes\sigma) & =(\varphi_{k}\otimes1)(\xi(e_{k-1}\otimes\sigma)+(-1)^{k}(\varphi_{k}\otimes1)\cdot T\cdot\xi(e_{k-1}\otimes\sigma)\\
 & =(-1)^{k}(\varphi_{k}\otimes1)\cdot T\cdot\xi(e_{k-1}\otimes\sigma)
\end{align*}
again, because $\varphi_{k}^{2}=0$ and $\varphi_{k}\circ\iota_{k}\circ\epsilon=0$.
We continue, using equation~\ref{eq:big-diag2} to compute $\xi(e_{k-1}\otimes\sigma)$:
\begin{align*}
\xi(e_{k}\otimes\sigma)= & (-1)^{k}(\varphi_{k}\otimes1)\cdot T\cdot\xi(e_{k-1}\otimes\sigma)\\
= & (-1)^{k}(\varphi_{k}\otimes1)\cdot T\cdot(\varphi_{k}\otimes1)\biggl(\xi(\partial e_{k-1}\otimes\sigma)\\
 & +(-1)^{k-1}\xi(e_{k-1}\otimes\partial\sigma)\biggr)\\
= & (-1)^{k}\varphi_{k}\otimes\varphi_{k}\cdot T\cdot\biggl(\xi(\partial e_{k-1}\otimes\sigma)\\
 & +(-1)^{k-1}\xi(e_{k-1}\otimes\partial\sigma)\biggr)
\end{align*}
If $k-1=0$, then the left term vanishes. If $k-1=1$ so $\partial e_{k-1}$
is $0$-dimensional then equation~\ref{eq:big-diag1} gives $f(\partial e_{1}\otimes\sigma)$
and this vanishes when plugged into $\varphi\otimes\varphi$. If $k-1>1$,
then $\xi(\partial e_{k-1}\otimes\sigma)$ is in the image of $\varphi_{k}$,
so it vanishes when plugged into $\varphi_{k}\otimes\varphi_{k}$.

In \emph{all} cases, we can write
\begin{align*}
\xi(e_{k}\otimes\sigma) & =(-1)^{k}\varphi_{k}\otimes\varphi_{k}\cdot T\cdot(-1)^{k-1}\xi(e_{k-1}\otimes\partial\sigma)\\
 & =-\varphi_{k}\otimes\varphi_{k}\cdot T\cdot\xi(e_{k-1}\otimes\partial\sigma)
\end{align*}
If $\xi(e_{k-1}\otimes\Delta^{k-1})=\eta_{k-1}\Delta^{k-1}\otimes\Delta^{k-1}$
(the inductive hypothesis), then 
\begin{multline*}
\xi(e_{k-1}\otimes\partial\sigma)=\\
\sum_{i=0}^{k}\epsilon_{k-1}\cdot(-1)^{i}[0,\dots,i-1,i+1,\dots k]\otimes[0,\dots,i-1,i+1,\dots k]
\end{multline*}
and the only term that does \emph{not} get annihilated by $\varphi_{k}\otimes\varphi_{k}$
is 
\[
(-1)^{k}[0,\dots,k-1]\otimes[0,\dots,k-1]
\]
 (see equation~\ref{eq:simplex-contracting-cochain}). We get
\begin{align*}
\xi(e_{k}\otimes\sigma) & =\eta_{k-1}\cdot\varphi_{k}\otimes\varphi_{k}\cdot T\cdot(-1)^{k-1}[0,\dots,k-1]\otimes[0,\dots,k-1]\\
 & =\eta_{k-1}\cdot\varphi_{k}\otimes\varphi_{k}(-1)^{(k-1)^{2}+k-1}[0,\dots,k-1]\otimes[0,\dots,k-1]\\
 & =\eta_{k-1}\cdot(-1)^{(k-1)^{2}+2(k-1)}\varphi[0,\dots,k-1]\otimes\varphi[0,\dots,k-1]\\
 & =\eta_{k-1}\cdot(-1)^{k-1}[0,\dots,k]\otimes[0,\dots,k]\\
 & =\eta_{k}\cdot[0,\dots,k]\otimes[0,\dots,k]
\end{align*}
where the sign-changes are due to the Koszul Convention. We conclude
that $\eta_{k}=(-1)^{k-1}\eta_{k-1}$.
\end{proof}

\section{\label{sec:Proof-of-theorem}Proof of 5.7}

If $\ring$ is a ring satisfying remark~\ref{assu:rdef}, let $\field$
denote its field of fractions (either $\ints_{p}$ or $\rats$). 

We begin with a general result:
\begin{lem}
\label{lem:diagonals-linearly-independent}Let $C$ be a free $\ring$-module
and let 
\[
\hat{C}=\ring\oplus\prod_{i=1}^{\infty}C^{\otimes i}
\]
--- where tensor-products are over $\ints$.

Let $e:C\to\hat{C}$ be the function that sends $c\in C$ to
\[
(1,c,c\otimes c,c\otimes c\otimes c,\dots)\in\hat{C}
\]
For any integer $t>1$ and any set $\{c_{1},\dots,c_{t}\}\in C$ of
distinct, nonzero elements, the elements 
\[
\{e(c_{1}),\dots,e(c_{t})\}\in\hat{C}\otimes\field
\]
are linearly independent over $\rfrac$. It follows that $e$ defines
an injective function
\[
\bar{e}:\ring[C]\to\hat{C}
\]
where $\ring[C]$ is the group-ring of $C$ (in which $C$ is regarded
only as an abelian group).\end{lem}
\begin{proof}
We will construct a vector-space morphism 
\begin{equation}
f:D=\hat{C}\otimes\field\to V\label{eq:diagonals-linearly-independent1}
\end{equation}
in several stages, with the property that the images, $\{f(e(c_{i}))\}$,
are linearly independent. 

Note that $\field\otimes_{\ints}\field=\field$ so that $\field\otimes C^{\otimes i}=(\field\otimes C)^{\otimes i}$
for all $i$. We begin with the ``truncation morphism''
\[
r_{t}:D\to\field\oplus\bigoplus_{i=1}^{t-1}\field\otimes C^{\otimes i}=\field\oplus\bigoplus_{i=1}^{t-1}(\field\otimes C)^{\otimes i}=D_{t-1}
\]
which maps $C^{\otimes1}\otimes\field$ isomorphically. Now we include
$D_{t-1}$ in the tensor algebra $T(C\otimes\field)$ and project
that onto the symmetric algebra $S(C\otimes\field)$:
\[
g:D_{t-1}\hookrightarrow T(C\otimes\mathbb{F})\twoheadrightarrow S(C\otimes\mathbb{F})=\mathbb{F}[b_{1},b_{2},\dots]
\]
where the set, $\{b_{i}\}$, is an $\ring$-basis for $C$, and the
target on the right is a polynomial ring with the $b_{i}$ as \emph{indeterminates.}
If $c\in C\otimes\mathbb{F}\subset D_{t-1}$, and $f=g\circ r_{t}$,
then $f(c)$ is a linear combination of the $b_{i}$. If $c_{1},\dots,c_{j}\in C$
with $j\le t$, then $c_{1}\otimes\cdots\otimes c_{j}\in C^{\otimes j}\subset D_{t-1}$
and 
\[
g(c_{1}\otimes\cdots\otimes c_{j})=g(c_{1})\cdots g(c_{j})\in\rfrac[b_{1},b_{2},\dots]
\]
It is not hard to see that 
\[
p_{i}=f(e(c_{i}))=1+f(c_{i})+\cdots+f(c_{i})^{t-1}\in\rfrac[b_{1},b_{2},\dots]
\]
 for $i=1,\dots,t$. Since the $f(c_{i})$ are \emph{linear} in the
indeterminates $b_{i}$, the degree-$j$ component (in the indeterminates)
of $f(e(c_{i}))$ is precisely $f(c_{i})^{j}$. It follows that a
linear dependence-relation
\[
\sum_{i=1}^{t}\alpha_{i}\cdot p_{i}=0
\]
with $\alpha_{i}\in\mathbb{F}$, holds if and only if
\[
\sum_{i=1}^{t}\alpha_{i}\cdot f(c_{i})^{j}=0
\]
 for all $j=0,\dots,t-1$. This is equivalent to $\det M=0$, where
\[
M=\left[\begin{array}{cccc}
1 & 1 & \cdots & 1\\
f(c_{1}) & f(c_{2}) & \cdots & f(c_{t})\\
\vdots & \vdots & \ddots & \vdots\\
f(c_{1})^{t-1} & f(c_{2})^{t-1} & \cdots & f(c_{t})^{t-1}
\end{array}\right]
\]
 Since $M$ is the transpose of the Vandermonde matrix, we get
\[
\det M=\prod_{1\le i<j\le t}(f(c_{i})-f(c_{j}))
\]
Since $f|C\otimes_{\ints}\rfrac\subset\hat{C}\otimes_{\ints}\rfrac$
is \emph{injective,} it follows that this \emph{only} vanishes if
there exist $i$ and $j$ with $i\ne j$ and $c_{i}=c_{j}$. The second
conclusion follows.
\end{proof}
This leads to the proof:
\begin{cor}
\label{cor:simplicial-abelian-injective}If $X$ is a pointed, reduced
degeneracy-free simplicial set, then then the morphism of Steenrod
coalgebras
\[
F_{X}:C(\pz X)\otimes\ring\to L_{\freeop}\left(C(X)\otimes\ring\right)
\]
in definition~\ref{def:pcmap} is injective.\end{cor}
\begin{proof}
If $E=C(\pz X)\otimes\ring$ and $C=C(X)\otimes\ring$, the results
of \cite{smith:cofree} imply that 
\[
L_{\freeop}E\subset E\oplus\prod_{n=0}^{\infty}\homzs n(\freeop(n),E^{\otimes n})
\]
and the map $\gamma_{X}$ in definition~\ref{def:pcmap} induces
a commutative diagram
\begin{equation}
\xyR{20pt}\xymatrix{{C(\pz X)}\ar[r]^{\alpha}\ar[rd]_{F_{X}} & {L_{\freeop}E}\ar[r]\ar[d]^{L_{\freeop}\gamma_{X}} & {\prod_{n=0}^{\infty}\homzs n(\freeop(n),E^{\otimes n})}\ar[d]^{\prod_{n=0}^{\infty}\homzs n(1,\gamma_{X}^{\otimes n})}\\
{} & {L_{\freeop}C}\ar[r] & {\prod_{n=0}^{\infty}\homzs n(\freeop(n),C^{\otimes n})}
}
\label{eq:proof-commut-dia}
\end{equation}
 where we follow the convention that $\homzs 0(\freeop(0),E^{0})=\ring$,
$\homzs 1(\freeop(1),E)=E$. Here
\begin{enumerate}
\item $\alpha$ is induced by the identity map of $E$ (regarded as a chain-complex),
\item $\gamma_{X}$ and $F_{X}$ are defined in definition~\ref{def:pcmap}.
\end{enumerate}
Let $p_{n}$ be projection to a factor
\[
p_{n}:\prod_{n=0}^{\infty}\homzs n(\freeop(n),E^{\otimes n})\to\homzs n(\freeop(n),E^{\otimes n})
\]
If $\sigma\in$ is an $m$-simplex defining an element $[\sigma]\in E_{m}$,
proposition~\ref{pro:simplicespropertyS} implies that 
\[
p_{2}\circ\alpha([\sigma])=\xi_{m}\cdot(e_{m}\mapsto[\sigma]\otimes[\sigma])\in\homzs 2(\rs 2,E\otimes E)
\]
where $\xi_{m}=(-1)^{m(m-1)/2}$ (see proposition~\ref{pro:simplicespropertyS})
and $\freeop(2)=\rs 2$.

Let $Z_{2}=e_{m}$ and $Z_{k}=\underbrace{e_{m}\circ_{1}\cdots\circ_{1}e_{m}}_{k-1\text{ iterations}}\in\freeop(k)$
be the operad-composite (see definition~\ref{def:operad-comps} and
proposition~2.17 of \cite{smith:1994}). The fact that operad-composites
map to composites of \emph{coproducts} (see proposition~\ref{prop:composition-coalgebra})
in a coalgebra implies that 
\[
p_{k}\circ\alpha([\sigma])=\xi_{m}^{k-1}\cdot(Z_{k}\mapsto\underbrace{[\sigma]\otimes\cdots\otimes[\sigma]}_{k\text{ factors}})\in\homzs k(\freeop(k),E^{\otimes k})
\]

If $\{\sigma_{1},\dots,\sigma_{t}\}\in\pz X$ are \emph{distinct}
$m$-simplices then $\{\gamma_{X}[\sigma_{1}],\dots\gamma_{X}[\sigma_{t}]\}\in C=C(X)\otimes\ring$
are \emph{also} distinct (although no longer \emph{generators}). 

Their images in $\prod_{n=0}^{\infty}\homzs n(\freeop(n),C^{\otimes n})$
will have the property that 
\begin{multline*}
p_{k}\circ F_{X}([\sigma_{i}])=\xi_{m}^{k-1}\cdot(Z_{k}\mapsto\underbrace{\gamma_{X}[\sigma_{i}]\otimes\cdots\otimes\gamma_{X}[\sigma_{i}]}_{k\text{ factors}})\\
\in\homzs k(\freeop(k),C(X)^{\otimes k})
\end{multline*}

Evaluation of elements of $\prod_{n=1}^{\infty}\homzs n(\freeop(n),C^{\otimes n})$
on the sequence $(\xi_{m}\cdot Z_{2},\xi_{m}^{2}\cdot Z_{3},\xi_{m}^{3}\cdot Z_{4},\dots)$
gives a homomorphism of $\ints$-modules
\[
j:\prod_{n=0}^{\infty}\homzs n(\freeop(n),C^{\otimes n})\to\prod_{n=0}^{\infty}C^{\otimes n}
\]
and $j\circ\gamma_{X}(\sigma_{i})$ is $e(\gamma_{X}[\sigma_{i}])$,
as defined in lemma~\ref{lem:diagonals-linearly-independent}. The
conclusion follows from lemma~\ref{lem:diagonals-linearly-independent}.
\end{proof}

\section{References}

\bibliographystyle{amsplain}

\providecommand{\bysame}{\leavevmode\hbox to3em{\hrulefill}\thinspace}
\providecommand{\MR}{\relax\ifhmode\unskip\space\fi MR }
\providecommand{\MRhref}[2]{%
  \href{http://www.ams.org/mathscinet-getitem?mr=#1}{#2}
}
\providecommand{\href}[2]{#2}


    \end{document}